\documentclass[10pt]{amsart}   	

\usepackage{geometry}                		
\usepackage{graphicx}				
\usepackage[parfill]{parskip}    		
							
\usepackage{mathrsfs}  
\usepackage{amsfonts}
\usepackage{amsthm}
\usepackage{mathtools}									
\usepackage{amssymb}
\usepackage{amsmath}
\usepackage{cite}
\usepackage{hyperref}
\usepackage{subfig}
\usepackage{caption}

\newtheorem*{rep@theorem}{\rep@title}
\newcommand{\newreptheorem}[2]{%
\newenvironment{rep#1}[1]{%
 \def\rep@title{#2 \ref{##1}}%
 \begin{rep@theorem}}%
 {\end{rep@theorem}}}
\makeatother

\newtheorem{theorem}{Theorem}[section]
\newtheorem{proposition}[theorem]{Proposition}

\newtheorem{corollary}[theorem]{Corollary}
\newtheorem{lemma}[theorem]{Lemma}
\newreptheorem{theorem}{Theorem}

\newtheorem{definition}[theorem]{Definition}
\newtheorem{example}[theorem]{Example}

\newtheorem{remark}[theorem]{Remark}

\begin{document}

\title[Divergence property of some generalised Thompson groups]
{Divergence property of the Brown-Thompson groups and braided Thompson groups}
\author{Xiaobing Sheng}

\address{Graduate School of Mathematical Sciences, 
	The University of Tokyo}

\email{sheng-xiaobing374@g.ecc.u-tokyo.ac.jp
}



\date{\today}							

\maketitle


\begin{abstract} 
Golan and Sapir proved that Thompson's groups $F$, $T$ and $V$ have linear divergence.
In the current paper, 
we focus on the divergence property of several generalisations of the Thompson groups. 
We first consider the Brown-Thompson groups $F_n$, $T_n$ and $V_n$ 
(also called Brown-Higman-Thompson group in some other context) 
and find that these groups also have linear divergence functions.
We then focus on the braided Thompson groups $BF,$ $\widehat{BF}$ and $\widehat{BV}$ and 
proved that these groups have linear divergence.
The case of $BV$ has also been done independently by Kodama.
\end{abstract}

\bigskip

\tableofcontents

\section{Introduction}

Thompson's groups $F$, $T$, $V$ were first constructed 
from a logic point of view 
while later found to have connections 
with many other branches of mathematics such as string rewriting systems, 
homotopy theory, 
combinatorics 
and dynamical systems.

$F$ and $T$ can be regarded as subgroups 
of the group of homeomorphisms of the circle 
while $V$ can be seen as a subgroup of the group of self-homomorphisms of the Cantor set. 
These groups have been generalised to many larger classes since they were defined. 
Thompson's group $V$ has first been generalised to 
a family of finitely presented infinite groups by Higman 
in the 70's 
and the groups $F$, $T$ and $V$ have further been extended to infinite families by Brown \cite{MR885095} 
and later being generalised by Stein \cite{MR1094555} 
and at the same time, 
some of their homological and simplicity results are obtained\cite{MR885095}. 
More recently, the braided version 
and the higher dimensional version ones are defined and investigated.

Among these generalisations of Thompson's groups, 
most of the early results are on algebraic and topological properties of groups. 
The geometry of the groups are beginning to attract the interest of more geometric group theorists
in the recent years 
partially because of the still open amenability problem of Thompson's group $F.$

We list a selection of the geometric results as follows: 
the result on the quadratic Dehn function of $F$ by Guba \cite{MR2104775, MR1750493}, 
the construction of the $CAT(0)$ cube complexes, 
on which the groups $F$, $T$, $V$ act properly by Farley \cite{MR2393179}, 
some subgroup distortion results by Burillo, et al. \cite{MR1670622, MR1806724, MR3091272, MR2452818}
and results by Golan and Sapir \cite{MR3978542} 
on the divergence property of the three groups.
A more recent result on the linear divergence of the braided Thompson's group $BV$ 
has been proved by Kodama \cite{Kodama:2020to}. 

\subsection{Background}

This work focuses on one of the central notion in geometric group theory, 
``hyperbolicity" or ``the nature of non-positive curvature'' of groups 
and we consider some generalisations of a particular 
class of groups, Thompson's groups.

Thompson's groups are not Gromov hyperbolic, 
but the above geometric results on the groups suggest that 
these groups might have some ``hyperbolicity"s in a coarse sense. 
Then the question is how we are going to measure these ``hyperbolicity", 
since there might be many generalisations of the notion of hyperbolic-type properties.
The one that we will be focusing on, in the current paper, 
is the divergence property.

The idea of divergence property first appeared
in \textit{Asymptotic invariants of infinite groups} by Gromov \cite{MR1253544}, 
where he suggested that in the non-positively curved spaces, 
\textit{One expects, 
that (at least after plusifications) `infinitely close' geodesic rays issuing from a point, 
either diverge linearly or exponentially}. 
The formal definition was first stated by Gersten \cite{MR1302334}, 
and the notion later proved to be a quasi-isometric invariant of metric spaces. 

Intuitively, divergence property measures \textit{how two geodesics in a metric space go apart,} 
or in other words, 
\textit{how hard it is to connect points on two distinct geodesic rays if the backtracking is not allowed.}
When considering the notion of the generalised ``hyperbolicity " or the hyperbolic-type properties, 
in some particularly ``nice" space ( for example, Gromov hyperbolic space, CAT(0) spaces), 
the condition of super-linear divergence is equivalent to 
the notion of having cut-points in the asymptotic cones in the groups or 
the ``Morse" property. 


When enlarging the scope further,
the divergence property also indicates some large scale geometric properties of the spaces 
and their associated groups.
Since the notion divergence is a quasi-isometry invariant, 
we would know precisely the quasi-isometry classes of the groups.

On the other hand, in Tran's thesis \cite{MR3474592},  
he investigated the divergence property of finitely presented groups, 
and he further defined relative divergence 
and the usual definition of divergence in the sense of Gersten \cite{MR1302334} 
which gives a lower bound on the divergence functions of all the geodesics in the whole space.  

In this paper, 
we focus on some generalisations of Thompson's groups 
and prove that similar divergence property can be found on geodesics of these groups. 
More precisely, 
we prove that Brown-Thompson groups and the braided Thompson group $BF$ 
have linear divergence property.
The result for $BF$ together with the result for $BV$ in \cite{Kodama:2020to} 
can be extended to obtain that 
$\widehat{BV}$ and $\widehat{BF}$ also have linear divergence.

\section{Basic definitions and notation}

Thompson's groups $F$, $T$ 
as well as the Brown-Thompson groups $F_n$, $T_n$ 
can be informally considered as subgroups of piecewise-linear homeomorphisms of 
the unit interval $[0,1]$ and the circle $S^1$, respectively, 
while the group $V$ and $V_n$ can be regarded as 
subgroups of the homeomorphisms of the Cantor set
and they all have several different interpretations. 
Here, we will be following the convention in \cite{MR3978542} 
and utilising both combinatorial and analytical definitions. 

\begin{definition}[Thompson's groups $F$, $T$]
Thompson's groups $F$, $T$, 
are the groups of orientation-preserving piecewise-linear homeomorphisms 
of the unit interval $ [0,1] $ and the $S^1$ where $S^1$ is the unit interval $[0,1]$ 
identifying the endpoints to themselves, respectively, 
which are differentiable except at finitely many dyadic rationals 
such that the slope of each subinterval is an integer power of $2.$
\end{definition}

\begin{definition}[The Brown-Thompson groups]
The Brown-Thompson groups $F_n$, $T_n$, 
are the groups of orientation-preserving piecewise-linear homeomorphisms 
of the unit interval $ [0,1] $ and the circle $S^1$, respectively, 
which are differentiable except at finitely many $n$-adic rationals 
such that the slope of each subinterval is an integer power of $n.$
\end{definition} 
\begin{definition}[Thompson's group $V$ and its generalisations $V_n$]
$V$ and $V_n$ are the groups of the right-continuous bijections 
from the unit interval $[0,1]$ onto itself 
which are differentiable except at finitely many dyadic and $n$-adic rationals, respectively, 
such that the slope of each interval is a power of $2$ and an integer power of $n,$ respectively. 
\end{definition}
From this standard interpretation of the groups, 
we define the following notion: 
an element $g \in V$ is \textit{supported} 
on some interval $(a, b) \subset [0,1]$ 
meaning that for any point $x \in (a, b),$ 
we have $g(x) \neq x.$ 
The collection of such subintervals in $[0,1]$ are called 
the \textit{support} of $g.$
The idea can be extended to the Brown-Thompson groups as well. 
The Thompson groups have many different definitions 
\cite{MR0396769, MR1426438, MR1806724}. 
Among these different interpretations of the elements of the Thompson groups, 
we will mainly use the tree pair representations. 

\subsection{The group elements, tree pair representations}
The tree pair representations come from the fact 
that intervals or circles with partitions by dyadic rationals 
can be identified with trees.

A pair of rooted $n$-ary trees having the same number of leaves 
with identical labelings give rise to a representation for an element of $F_n$, 
a pair of such trees with cyclic labelings 
represent an element of $T_n$ \cite{Sheng:2018aa} 
and 
a pair with labelings 
induced from an element of the symmetric group $S_{l(n-1)+1}$ representing an element of $V_n,$ 
here $l \in \mathbb{N}$ and $l(n-1)+1$ is the number of the leaves in each tree of the pair.
Hence, the elements in Thompson's groups and the generalisations 
can be represented by such labeled tree pairs. 

Such pairs of trees can be used to give 
the unique normal form for elements in $F, F_n$ 
and a standard form for elements in $T, T_n$ 
with respect to the infinite standard generating sets. 
These forms can be used to estimate the word lengths of the elements of these groups 
and the bounds of the word lengths of elements in $V, V_n$
by replacing the standard infinite generating sets by the finite ones.
This alternative definition is also given in \cite{MR1806724, MR2452818}. 
Below is a more precise description for the case of the group $F.$

%
For each unit interval $[0,1]$ 
with finitely many dyadic breakpoints, 
the subintervals are of the form 
$[ \frac{k}{2^m}, \frac{k+1}{2^m}]$ 
where $k+1\leq 2^m$ and $m, k \in \mathbb{N}\cup \{0\}$.
A unit interval $[0,1]$ with finitely many breakpoints are called 
\textit{dyadically subdivided interval}
and can be associated with a rooted finite binary tree 
by associating each of these breakpoints with a $2$-caret, 
i.e. a rooted binary tree with one vertex to be the root and two edges attached to it.
Then we can associate two intervals 
having the same number of dyadic breakpoints 
with a pair of finite binary trees 
with the same number of leaves,
namely, we have associated an element of Thompson's group $F$
with a pair of rooted finite binary trees.

Here the leaves in each tree of a tree pair 
correspond to the dyadic subintervals in each $[0,1]$
and they are labeled by the natural number $\mathbb{N}$ 
to indicate 
which leaf in the source tree maps to which leaf in the target.
It is proved in \cite{MR1426438} that 
there is a uniquely reduced tree pair representative for each element in $F.$ 
The above construction can be generalised to $T$ and $V$ 
with some variations on the labelings 
as well as to the Brown-Thompson groups $F_n$, $T_n$ and $V_n$ 
by changing binary trees to the $n$-ary trees.

The Brown-Thompson groups $F_n$ $T_n$ and $V_n$ 
as the groups $F,$ $T,$ $V$ 
are proved to be finitely presented 
\cite{MR1426438, MR885095}. 
Here we illustrate the tree pair representations of the group elements 
in the standard finite generating sets $\mathcal{C}$ 
to give a brief idea of this form in Figure \ref{fig:generators}.
\begin{figure}[h!]
\centerline{\includegraphics[width=5in]{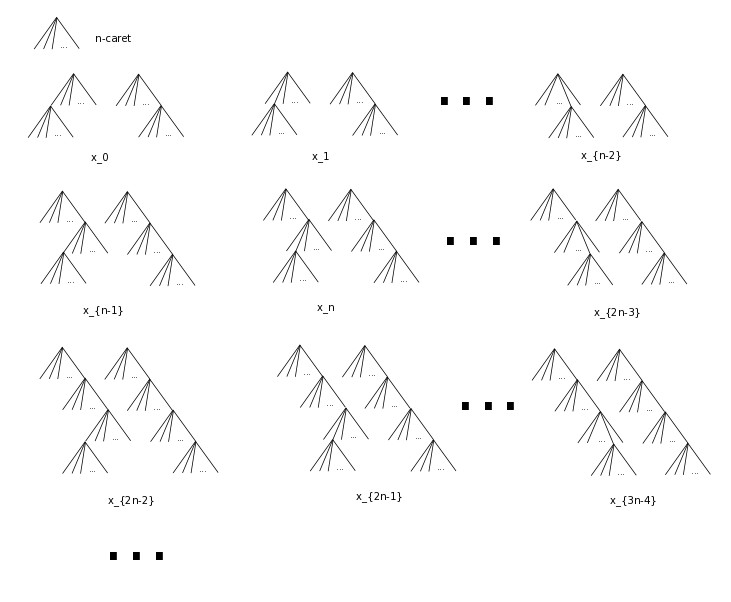}}
\centerline{\includegraphics[width=5in]{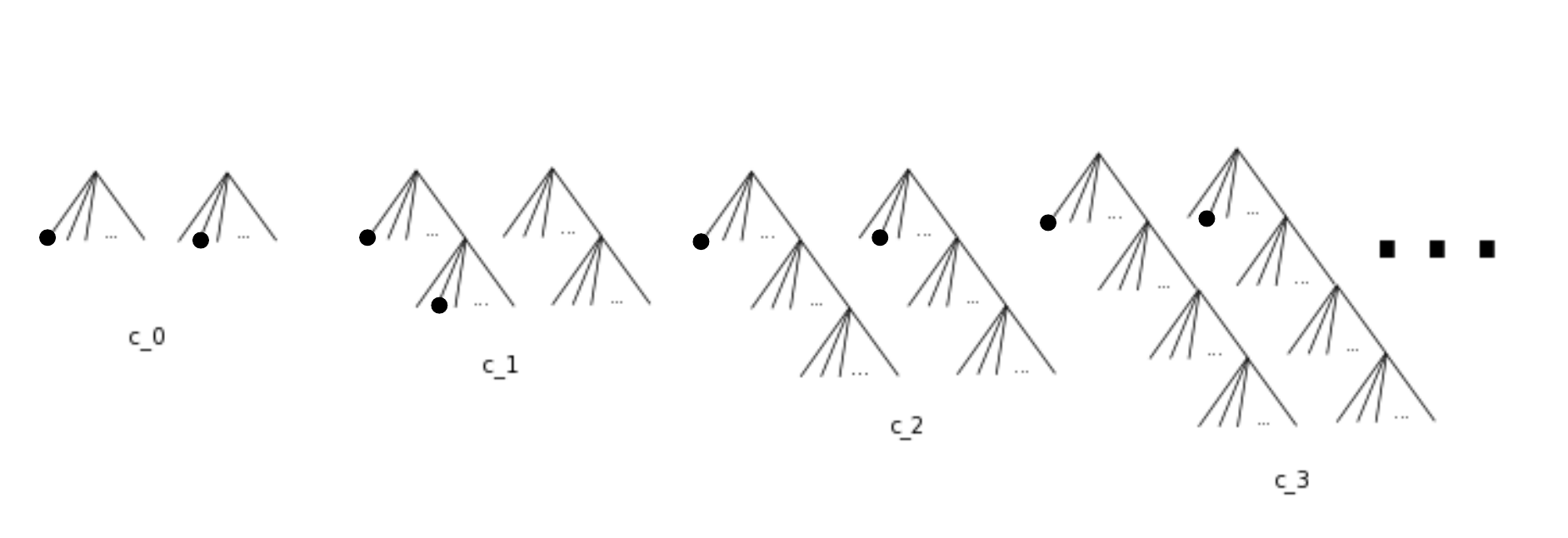}}
\centerline{\includegraphics[width=3in]{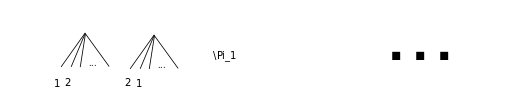}}
\caption{The infinite generating set for $F_n$, $T_n$ and $V_n$  \label{fig:generators} }
\end{figure}
 
\subsection{Tree pairs, the Cantor set and the binary words}
Alternatively, 
when we consider such tree pairs, 
we consider 
the correspondence between a rooted finite binary tree 
and a different rooted finite binary tree 
with the same number of leaves 
which can be regarded as a partial automorphism of the infinite binary tree \cite{MR2952772}
such that only the finite top part from the root of the tree is replaced. 
In addition, 
since the boundary of the infinite binary tree
can be identified with 
the Cantor set, 
an element of Thompson's group $V$ can be regarded as 
the homeomorphisms of the Cantor set 
that preserves finite number of the partitions of the Cantor set.

We consider again the tree pair representations 
as partial automorphisms of the rooted infinite binary tree denoted by $\mathcal{T}_{\infty}.$ 
For $\mathcal{T}_{\infty},$ 
each node $\nu$ can be associated with a unique finite path from the root, 
and we can read off a binary word from the root along the path to the node $\nu$ as follows:
we identify the root of the binary tree with the empty word $\varepsilon,$ 
when the path going down to the lower left vertex through an edge in a $2$-caret, 
we add $0$ to the word representing the path 
and when the path going down to the lower right vertex in the same $2$-caret, 
we concatenate a $1$ to the word representing the path.
Then, each leaf in a rooted finite binary tree can be identified 
with a finite binary word in alphabet $\{0,1\}.$ 
The correspondence between a finite binary tree and another finite binary tree 
with the same number of the leaves 
can then be regarded as 
the correspondence between a finite set of binary words 
and another finite set of binary words of the same order.  
These lead to the concept of branches.

\subsection{Branches and branches of a tree}
\label{subsec22}
Following the description in Golan and Sapir \cite{MR3978542}, 
we let the symbol $( \mathrm{T}_{+}, \sigma, \mathrm{T}_{-} )$ denote a tree pair 
defining an element $g \in V $, 
where $ \mathrm{T}_{+}$ and $\mathrm{T}_{-} $ are rooted finite binary trees 
with the same number of the leaves 
and $\sigma$ denotes the permutation of the order of the leaves.
\begin{definition}
For every leaf in some finite rooted binary tree $\mathrm{T},$ 
we associate with it 
a path labelled by a binary word 
from the root to $\gamma_o$ with the notation 
$\ell_{\gamma_o}(\mathrm{T}).$ 
We call the path $\ell_{\gamma_o}(\mathrm{T})$ a branch of the tree. 
\end{definition}
By identifying a leaf in the finite binary tree 
with a finite path 
and hence with a finite binary word $\omega_o$ over the alphabet $\{0,1\},$ 
we could express the correspondence between a leaf labeled $\omega_o$ in source tree $\mathrm{T}_{+}$ 
and a leaf labeled $\omega_{\sigma(0)}$ in the target tree $\mathrm{T}_{-}$ 
of a tree pair $( \mathrm{T}_{+}, \sigma, \mathrm{T}_{-} )$ 
representing the group element $g$ of $V$ 
by the correspondence between a finite path $\ell_{\omega_o}(\mathrm{T}_{+})$ 
in the source tree $\mathrm{T}_{+}$ 
and a finite path $\ell_{\omega_{\sigma(o)}}(\mathrm{T}_{-})$ 
in the target tree $\mathrm{T}_{-}.$ 
\begin{definition}
We denote the correspondence by 
$\ell_{\omega_o}(\mathrm{T}_{+}) \mapsto \ell_{\omega_{\sigma(o)}}(\mathrm{T}_{-})$ 
for $( \mathrm{T}_{+}, \sigma, \mathrm{T}_{-} )$ representing $g \in V$ 
and call it a branch of tree pair $( \mathrm{T}_{+}, \sigma, \mathrm{T}_{-} )$ 
or a branch of the group element $g$ that the tree pair represents. 
\end{definition}
We can now generalise the notion to as follows: 
we denote, similarly, by $( \mathrm{T}_{n(+)}, \sigma, \mathrm{T}_{n(-)} ),$ 
a tree pair where $\mathrm{T}_n$ is an $n$-ary tree, for an element $g \in \mathcal{G}.$
Here $\mathcal{G}$ could be $F_n$, $T_n$, $V_n$ where $n \in \mathbb{N}$ and $n \geq 2.$ 
When $\mathcal{G} = F_n$, $\sigma$ represents trivial permutation, 
when $\mathcal{G} = T_n$, $\sigma$ is a cyclic permutation in the cyclic group of order $l(n-1)+1$ 
where $l \in \mathbb{N}$ and $\sigma$ represents 
a permutation in the symmetric group $\mathfrak{S}_{l(n-1)+1}$ 
for some $l \in \mathbb{N}$ when $\mathcal{G} = V_n.$

\begin{definition}
Then we could associate a path labelled by an $n$-ary word 
from the root of the tree to some leaf $\omega_{o}$ in the source tree
to the notation $\ell_{\omega_{o}}(\mathrm{T}_n).$ 
Similarly, we denote by $\ell_{\omega_{o}}(\mathrm{T}_{n(+)}) \mapsto \ell_{\omega_{\sigma(o)}}(\mathrm{T}_{n(-)})$ 
some branch of the element $g \in \mathcal{G}.$
\end{definition}

Note that the words corresponding to the leaves are not binary words anymore, 
they are now some $n$-ary words.

The branches of a tree pair representation 
provide some partial information of a group element.
For instant, when considering the group elements of the Thompson groups 
to be the maps from the unit interval to itself, 
a branch of the unique tree pair representing a group element mapping 
a branch of the source tree to a branch of the target tree 
only gives the map between two subintervals 
which reflects the self-similarity properties of the trees. 
More precisely, when we identify the unit interval $ [0,1]$ with an $n$-nary tree $\textrm{T}_{n(\pm)}$, 
a branch of a tree can be interpreted as a subinterval of $ [0,1]$ with $n$-adic bounds 
and a branch of a tree pair can be seen as 
the correspondence between two subintervals of the source interval $[0,1]$ 
and the target interval $[0,1]$ respectively.

When given the product of two group elements represented by tree pairs, 
the support, branches and the tree pair representation of the product 
can be interchanged to provide desired information. 

We introduce another notation before we connect the dots together.
\begin{definition}
Let $g \in V_n,$ 
represented by tree pair $(\mathrm{G}_{+}, \sigma_g, \mathrm{G}_{-}).$ 
We denote by $g_{[ \omega' ]}$ the tree pair constructed as follows:
Take an element $h \in V_n$ represented by the tree pair 
$(\mathrm{T}_{+}, \sigma, \mathrm{T}_{-})$ containing a branch 
$\ell_{\omega'}(\mathrm{T}_{+}) \mapsto \ell_{\omega'}(\mathrm{T}_{-}).$
We attach the tree $\mathrm{G}_{+}$ to the leaf labeled 
$\omega'$ in the branch $\ell_{\omega'}(\mathrm{T}_{+})$ 
and $\mathrm{G}_{-}$ to the leaf labeled 
$\omega'$ in the branch $\ell_{\omega'}(\mathrm{T}_{-}).$
More generally, 
let $t \in V_n$ be represented by 
$(\mathrm{T}_{+}, \sigma, \mathrm{T}_{-})$
and let $\ell_o(\mathrm{T}_{+}) \mapsto \ell_{\sigma(o)}(\mathrm{T}_{-})$ be some branch of the element $t.$ 
Then we denote by $g_{[\ell_o(\mathrm{T}_{+}) \mapsto \ell_{\sigma(o)}(\mathrm{T}_{-})]}$ 
the tree pair 
such that the source tree is the tree $\mathrm{T}_{+}$ with $\mathrm{G}_{+}$ attached to the leaf $o$ 
and the target tree is the tree $\mathrm{T}_{-}$ with $\mathrm{G}_{-}$ attached to the leaf $\sigma(0).$
\end{definition}
%
%
For example, 
we could take both the notation $x_{0_{[u+(n-1)]}}$ and $x_{{n-1}_{[u]}}$ 
where $u \in \{0, \cdots, n-2\}$ 
to be the same tree pair representation 
to represent the same element.
\label{branch}

This simply connects the dynamical interpretation 
and the combinatorial interpretation of the group elements of Thompson's groups \cite{MR3091272} 
as well as provides a method to estimate of the changes in the number of carets.

Combining these different interpretations of the elements of the Thompson groups, 
we realised that 
the branches of tree pair representation of the group elements 
keep track of the support of the groups 
and detect the changes in the word length 
when having the product of the group elements. 
We will mainly use the tree pair representations and the branches 
to find the desired path for some later arguments.

\subsection{The estimate of the word lengths via tree pair representations} 
In the current section, 
we investigate the word lengths of the elements in the Brown-Thompson groups 
and it turns out that the word length relies largely on the number of the carets 
in the tree pair representation. 

The \textit{word length} of a group element $g$
with respect to some finite generating set of this group is the length of the shortest word 
representing $g$ 
in this finite generating set 
denoted by $| \cdot |$, or $| \cdot |_{\Sigma}$ 
when the finite generating set $\Sigma$ needs to be specified. 

Since the tree pair representations are usually not unique 
for a group element in Thompson's groups and their generalisations, 
we take the unique reduced tree pair representative for each group element \cite{MR1806724}, 
define the following form 
with respect to the infinite generating set.

\begin{definition}[\cite{MR2452818,Sheng:2018aa}]
Let the reduced labelled tree pair $(\mathrm{T}_{n(+)}, \sigma, \mathrm{T}_{n(-)} )$ represent 
an element $g$ in $T_n$ or in $V_n$ 
and $\mathrm{T}_{n(+)}$ and $\mathrm{T}_{n(-)}$ each have $i$ carets. 
Let $\mathrm{R}$ be the all right tree with $i$ carets,
i.e. a tree constructed by attaching carets consequently 
to the rightmost leaf of the previously attached caret.  
We write $g$ as a product $\textbf{p}\sigma\textbf{q},$ 
where:

\begin{enumerate}

     \item $\textbf{p}$, a positive word in the infinite generating set of $F_{n}$ 
     and of the form $\textbf{p} = x_{i_1}^{r_1}x_{i_2}^{r_2} \cdots x_{i_y}^{r_y} $ where $i_1 < i_2 \cdots < i_y $) 
     with tree pair $(\mathrm{T}_{n(+)},\mathrm{id}, \mathrm{R}).$
     \item

           \begin{itemize} 
           
           \item For $g \in F_n,$ $\sigma$ is just the identity $\rm id,$
           
           \item For $g \in T_n$, $\sigma$ is $c_{i-1}^{j},$ 
           the $j$th power of torsion element $c_{i-1}$ with $1 \leq j < i(n-1)+1$, 
           which can be represented by a tree pair with two $\mathrm{R}$'s with $i$ $n$-carets, 
           i. e. $(\mathrm{R}, c_{i-1}^{j}, \mathrm{R}).$ 

           \item For $g \in V_n$, $\sigma$ is a permutation in the symmetric group $\mathfrak{S}_{i(n-1)+1},$
           which permutes the leaves of the tree pair, 
           we have $(\mathrm{R}, \sigma, \mathrm{R})$ as the factor in the middle,

           \end{itemize}

           and  
      \item $\textbf{q}$, a negative word, 
is the of the form $x_{j_z}^{-s_z} \cdots x_{j_2}^{-s_2}x_{j_1}^{-s_1}$ where $j_1< j_2 <\cdots < j_z$) 
for an element in $F_{n}$ represented by  $(\mathrm{R},\mathrm{id}, \mathrm{T}_{n(-)}).$

\end{enumerate} 

We call such a product $\textbf{p}\sigma\textbf{q}$, a generalised $\textbf{pcq}$ factorisation. 
An example is illustrated in Figure \ref{fig:fig2}.
\end{definition}


\begin{figure}[h]
\centerline{\includegraphics[width=5in]{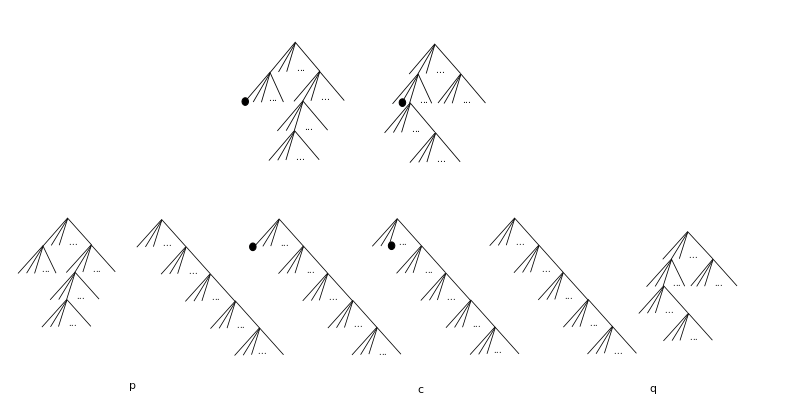}}
 \caption{$\textbf{pcq}$ factorisation  }
 \label{fig:fig2}
 
 \end{figure}

Note that a word in the form of 
$x_{i_1}^{r_1}x_{i_2}^{r_2} \cdots x_{i_y}^{r_y}c_{i-1}^{j} x_{j_z}^{-s_z} \cdots x_{j_2}^{-s_2}x_{j_1}^{-s_1} $ in $T_n$, 
where $i_1 < i_2 \cdots < i_y$ and $j_z >\cdots > j_2 > j_1 $, $r_k$'s 
and $s_l$'s are all positive and $1 \leq j \leq i(n-1)+1$ for the cyclic part $c_{i-1}^{j}$ in the middle.
$x_{i_1}^{r_1}x_{i_2}^{r_2} \cdots x_{i_y}^{r_y}$ and $x_{j_z}^{-s_z} \cdots x_{j_2}^{-s_2}x_{j_1}^{-s_1}$ 
are words represented in the standard infinite generating set of $F_n$ 
Figure \ref{fig:generators}).

\subsubsection{The estimate of the word length of elements in $F_n$, $T_n$}

There is a connection between the word length 
and the uniquely reduced tree pair representation of a group element 
by comparing the number of the building blocks of each $n$-nary trees in the tree pair, 
which is a rooted $n$-ary tree with one vertex representing the root and $n$ edges attached to it, 
called the $n$-caret, 
and the number of the leaves in each tree. 

As is proved in Burillo et. al \cite{MR1806724, MR2452818, Sheng:2018aa}, 
these two notion are interchangeable 
and the word length of the group elements in $F$, $F_n$, $T$, $T_n$ can be estimated 
by the number of carets of the trees in the tree pair representation as follows.

\begin{lemma}[\cite{MR2452818}]
Let $\omega \in T_n,$ 
Let $N_n(\omega)$ denote the number of carets in each of the trees in the unique tree pair representing $\omega$, 
let $|\omega |_n$ denote the word length of $\omega$ 
with respect to the standard finite generating set $\{x_0, x_1, \cdots, x_{n-1}, c_0\}$ of the group $T_n,$ 
then there exists a constant $C'$ such that the following estimate is satisfied for any $\omega$,
   $$\frac{N_n(\omega)}{C'} \leq | \omega |_n \leq C'N_n(\omega)$$
 \label{lem1}  
\end{lemma}
Thus the word length can also be estimated by the number of leaves 
in each tree of the unique tree pair representation.
\subsubsection{The definition of $V_n$ from a logic viewpoint}
For group $V,$ 
the word length of the group elements are proved to have a lower bound comparable 
with the number of leaves and the number of carets in Birget \cite{MR2104771} 
in a slightly more combinatorial description. 
We follow the method from \cite{MR2104771} 
and generalise the argument to $V_n$ 
which requires the definition of the Thompson groups through coding theory and 
this, in fact, coincides with the original definition of Thompson's groups.

We briefly set up the foundation as follows.
Let $A$ be a finite alphabet, 
given any two elements $u, v$ in the set $A^{\ast}$ of the words, 
we denote by $u \cdot v$ the concatenation 
and say that 
$u$ is a prefix of $v$ if and only if $v = ux$ for some $x \in A^{\ast}.$
This provides a partial order for the words.
A \textit{prefix code} over $A$ is defined to be a subset of $A^{\ast}$ 
such that no elements of this set is the strict prefix of any other elements of the set. 
For $A,$ 
we define the \textit{right ideal} of $A^{\ast}$ to be $R \subseteq A^{\ast}$ 
such that $R \cdot A^{\ast} \subseteq R.$ 
A set $\Gamma \subseteq R$ is called a \textit{right ideal generator} of $R$ 
if and only if $R = \Gamma \cdot A^{\ast}$. 
$R$ is called \textit{essential} if and only if 
$R$ has a nonempty intersection with every other right ideal of $A^{\ast}.$ 
A \textit{right ideal homomorphism} of $A^{\ast} $ is a function $\phi:R_1\rightarrow R_2$ 
such that $R_1$ and $R_2$ are nonempty right ideals of $A^{\ast}$ 
and such that for all $u \in R_1$ and for all $x \in A^{\ast} : \phi (u) \cdot x = \phi(ux)$ 
and it is a bijective right ideal homomorphism. 
An \textit{extension} of a right ideal homomorphism 
$\phi:R_1 \rightarrow R_2$ is a right ideal homomorphism 
$\Phi : J_1 \rightarrow J_2$ where $J_1, J_2 $ are right ideals 
such that $R_i \subseteq J_i$ for $i \in \{1, 2\}$ and $\Phi$ agrees with $\phi$ on $R_1.$ 
A right ideal homomorphism is said to be \textit{maximal} if and only if it has no strict extension in $A^{\ast}.$ 
Then we have all the basic ingredients for defining $V$ and it's generalisation alternatively.

\begin{definition}[\cite{MR2104771}]
$V$ is the partial action group on the set of words $\{a, b\}^{\ast}$ 
consisting of all maximal isomorphisms 
between the finitely generated essential right ideals of $\{a, b\}^{\ast}.$ 
The group multiplication of two elements of $V$ is 
the maximum extension of the composition of two elements of the group.
\end{definition}
Similarly, we define $V_n$.
\begin{definition}[The generalised Thompson group $V_n$]
$V_n$ is defined as 
the partial action group on the set of words $\{a_0, a_1, \cdots, a_{n-1} \}^{\ast}$ 
consisting of all maximal isomorphisms 
between the finitely generated essential right ideals of $\{a_0, a_1, \cdots, a_{n-1}\}^{\ast}.$ 
The group multiplication of two elements of $V_n$ is defined similarly as 
the maximum extension of the composition of two elements of the group.\end{definition}

Need to mention that $F_n$ can be defined as the subgroup of $V_n$ 
consists of all right ideal isomorphism of $\{a_0, a_1, \cdots, a_{n-1} \}^{\ast}$ that 
preserves the dictionary order as for $F$ in \cite{birget2004groups}. 

\begin{definition}[\cite{MR2104771}]
For a right-ideal isomorphism $\phi:P_1A^{\ast} \rightarrow P_2A^{\ast}$ (a bijection), 
where $P_1$ and $P_2$ are finite maximal prefix codes, 
the restriction $P_1\rightarrow P_2$ of $\phi$ is called the table of $\phi.$ 
Define the table size $\lVert \phi \rVert$ to be $| P_1| $ $( =  | P_2 | ).$
For an element $g \in V$, 
the table size $\lVert g \rVert$ of $g$ is defined to be 
the table size of the maximally extended right-ideal isomorphism that represents $g$.
\end{definition}

\subsubsection{An estimate of the word length of elements in $V_n$ from the logic viewpoint}

The following lemma indicates that the table size and the word length of an element of $V$ are comparable. 
\begin{lemma}[ \cite{MR2104771}] 
The table size and the word size of an element $g \in V$ are related as follows: 

\begin{enumerate}

\item There are $c_{\Delta}, c'_{\Delta} > 0 $ 
(depending on the choice of finite generating set $\Delta$) 
such that for all $g \in V - \{1\}$: 
$c'_{\Delta} \lVert g \rVert \leq  | g |_{\Delta}  \leq c_{\Delta} \lVert g \rVert \cdot \log_2 \lVert g \rVert . $  

\item For almost all $g \in V$, $| g |_{\Delta} > \lVert g \rVert \cdot \log_{2 | \Delta |} \lVert g \rVert  .$

\end{enumerate} 

``Almost all" means that in the set $ \{g \in V: \lVert g \rVert = n \},$
the subset that does not satisfy the above inequality has a proportion that tends to $0$ exponentially fast as $n \mapsto \infty.$

\end{lemma}
The analogue for $V_n$ is as follows.

\begin{lemma}[The word length of elements in $V_n$]
The table size and the word size of an element $g \in V_n$ are related as follows: 

There are $c_{\Delta}, c'_{\Delta} > 0 $ 
(depending on the choice of finite generating set $\Delta$) 
such that for all $g \in V - \{1\}$: 
$c'_{\Delta} \lVert g \rVert \leq  | g |_{\Delta}  \leq c_{\Delta} \lVert g \rVert \cdot \log_2 \lVert g \rVert . $  
\label{lem210}
\end{lemma}

\begin{proof}[Proof of Lemma \ref{lem210}]
We extend \cite{MR2104771} to find a canonical factorisation for elements in $V_n$ 
to a right-ideal automorphism and two elements of $F_n$, 
then prove the latter to have bounded word length.

The ideal comes naturally from Thompson's orginal definition. 
When fixing the maximal prefix code 
$S_{\lVert g \rVert }$ of some cardinality $l(n-1)+1,$ 
for an group element $g$ of $V_n$,
there exists a unique decomposition 
such that $g = \beta_g\pi_g\alpha_g$ 
where  $\alpha_g, \beta_g \in F_n$ and $\pi_g$ is a automorphism 
whose table is a permutation of the $S_{\lVert g \rVert }$.
This is because, 
by considering the group element as a maximal right-ideal isomorphism $\phi$,
$\alpha_g$ and $\beta_g$ can be uniquely defined as order-preserving bijective, 
hence being uniquely determined by the maximal prefixed codes. 
This factorisation coincides with the $\textbf{pcq}$ factorisation.

Next, we look at the word length of $\alpha_g$ and $\beta_g$. 
The linear boundedness of the word length of the group elements in $F_n$ 
can either be deduced from \cite{MR1806724} 
or from a similar argument as in \cite{MR2104771}.  
They are essentially the same estimate.

Finally, we focus only on $\pi_g$, 
which can be regarded as a table of some automorphisms 
from the finite maximal prefix code $S_{\lVert g \rVert} $ to itself 
and hence is a permutation with bounded number of transpositions.
The transpositions defined in a generalised version of the one in \cite{MR2104771} 
give exactly the process of the permutation of the labels 
on the two identical trees in the tree pair representation.
\end{proof}
By the above argument, 
we also have a lower bound for estimating the word length of the elements in $V_n$.


\section{Divergence functions of the generalised Thompson's groups (the Brown-Thompson groups)}

\subsection{The Divergence function}

For metric spaces, 
there are several different types of divergence functions\cite{ MR1302334, Dru_u_2009,MR3978542,MR3474592}.
The intuitive idea of divergence property is to understand 
to what extend two geodesics starting from the same point in a metric space go away from each other.
Here we are taking the convention from \cite{Dru_u_2009}.


The divergence of the group is narrowed down 
from the divergence of the metric spaces 
by considering the divergence function of the Cayley graph of the groups 
with respect to the standard finite generating set defined below in Subsection \ref{stg}. 
From now on, 
we use $| \cdot |$ as the notation for the word length of the group elements.

\begin{definition}[\cite{Dru_u_2009}]
Define $div(a,b,c; \delta)$ as the infimum of the length of the paths 
for a geodesic metric space $(X, d )$ connecting $a$, $b$ 
and avoiding the ball $B(c, \delta r)$ centered at $c$ with radius $r$, 
where $a$, $b$, $c$ are points taken arbitrarily in the metric space $X,$
and $r$ is the minimum of the distances between $c$ and $a$ or between $b$ and $c.$
\end{definition}

\begin{definition}[Divergence function \cite{Dru_u_2009}]
The divergence function $Div(m; \delta)$ is the maximum of all $div(a,b,c; \delta)$ 
such that the pairs $(a, b)$ are within distance $m$,
 i.e. $dist(a, b) \leq m.$
\end{definition}
The goal here is to find the shortest path between two points at distance $n$ from the identity 
in the Cayley graph of the generalised Thompson groups 
and compare this path to some functions.

\subsection{Divergence property of $F_n, T_n, V_n$ }
As is in the Subsection \ref{subsec22}, 
we denote by $\mathcal{G}$ any of the groups $F_n$, $T_n$, $V_n.$
For an element $g \in \mathcal{G}$, 
we denote by $\mathcal{N}_{\mathcal{G}}(g)$ the number of leaves in each of the trees in the reduced tree pairs representing $g$ in group $\mathcal{G}$ or $\mathcal{N}(\cdot)$ when the group is clear. 

\begin{theorem}
There exist constants $\delta, D > 0$ such that the following holds. 
Let $g_1,g_2 \in \mathcal{G} $ be two elements with $\mathcal{N}(\cdot) \geq 3n-2.$ 
Then there is a path of length at most $D(|g_1|+|g_2|)$ in the Cayley graph $\Gamma = Cay(\mathcal{G}, X)$ 
which avoids a $ ( \delta\min\{| g_1| , | g_2| \} )$-neighbourhood of the identity 
and which has initial vertex $g_1$ and terminal vertex $g_2$.
\label{thm1}
\end{theorem}
We will prove Theorem \ref{thm1} over the course of this section.
\begin{corollary}
The generalised Thompson's groups $F_n$, $T_n$ and $V_n$ have linear divergence.
\label{thm2}
\end{corollary}
\begin{proof} 
This is a direct consequence of Theorem \ref{thm1}.
\end{proof}
The next Proposition which can be regarded as an analogue of \cite[Prop2.1]{MR3978542} 
can either be proved by the estimate on the number of carets as in 
\cite{MR3978542, MR2452818, Sheng:2018aa}
or using the estimate by computation in \cite{MR2104771} 
and the analogues of \cite{MR2104771} in the discussion above.

Let $\mathcal{A}_n = \{ x_0, x_1, \cdots, x_{n-1}\}$, 
$\mathcal{B}_n = \{ x_0, x_1, \cdots, x_{n-1}, c_0\} $, 
$\mathcal{C}_n = \{ x_0, x_1, \cdots, x_{n-1}, c_0, \pi_0\}$  
be the \textit{standard generating sets} of $F_n$, $T_n$, $V_n$, respectively, 
$\mathcal{N}_{\mathcal{G}}(g)$ be the number of leaves as above.
\label{stg}

\begin{proposition}[]
There exist constants $ 0 < c < 1$ and $C > 1$ such that the following holds
\begin{enumerate}
\item For every $g \in F_n$ we have $c \mathcal{N}_{F_n}(g) \leq | g |_{\mathcal{A}} \leq C \mathcal{N}_{F_n}(g)$.
\item For every $g \in T_n$ we have $c \mathcal{N}_{T_n}(g) \leq | g |_{\mathcal{B}} \leq C \mathcal{N}_{T_n}(g)$.
\item For every $g \in V_n$ we have $c \mathcal{N}_{V_n}(g) \leq | g |_{\mathcal{C}} $. 
\end{enumerate}
\label{prop1}
\end{proposition}

\begin{proof}[Proof of Proposition \ref{prop1}]
Since the number of carets and the number of the leaves of the unique tree pair 
representing an element in Brown-Thompson groups $F_n$ and $T_n$ 
are interchangeable \cite{MR1806724, Sheng:2018aa}, 
together with Lemma \ref{lem1}, 
the first two statements for the groups $F_n$ and $T_n$ are satisfied. 

A lower bound for the word length of the group elements of $V_n$ can be given 
by a similar argument as in \cite{MR2104771} 
in the proof of Lemma \ref{lem210}. 
By defining the \textit{table size} for $V_n,$
we could prove that the table size and the word length of an element in $V_n$ are related, 
and hence find a bound.
\end{proof}

The next step is to find out 
how concatenating one extra generator may change the number of the carets. 
\begin{lemma} 
For $g \in \mathcal{G}$ being represented 
by some reduced tree pair representation $(\mathrm{G}_{n(+)},\sigma_g, \mathrm{G}_{n(-)})$ 
with at least three carets in each tree, 
let $\bar{\ell}_{\gamma}(\mathrm{G}_{n(-)})$ denote the length of the branch of the target tree $\mathrm{G}_{n(-)},$ 
where $\gamma$ is an $n$-ary word corresponding a path from the root in $\mathrm{G}_{n(-)}.$ 
Denote by $(\mathrm{GX_i}_{n(+)},\sigma_{gx_i}, \mathrm{GX_i}_{n(-)})$  
the reduced tree pair representation of $gx_i \in \mathcal{G}.$ 
We then have 
$$\mathcal{N}(g) - (n-1) \leq  \mathcal{N}(gx_i^{\pm1}) \leq \mathcal{N}(g) + (n-1) $$
  where $i \in \{0, \cdots, n-2\}$. 
\begin{enumerate}
\item If $\bar{\ell}_{0}(\mathrm{G}_{n(-)}) = 1,$
then $\mathcal{N}(gx_i^{+1}) = \mathcal{N}(g) +(n-1),$ 
and $\bar{\ell}_{0}(GX_{i}^{+1}) = 1.$ 
\item If $ \bar{\ell}_{0}(\mathrm{G}_{n(-)}) \neq 1,$ 
then $\mathcal{N}(gx_i^{+1}) = \mathcal{N}(g),$ 
or $\mathcal{N}(gx_i^{+1}) = \mathcal{N}(g) - (n-1),$ 
Moreover, $\bar{\ell}_{i}(\mathrm{GX_i}_{n(-)}) =  \bar{\ell}_{i}(\mathrm{G}_{n(-)}) -1.$
\item If $\bar{\ell}_{i}(\mathrm{G}_{n(-)}) = 1$, 
then $\mathcal{N}(gx_i^{-1}) = \mathcal{N}(g) +(n-1),$
$\bar{\ell}_{j}(\mathrm{GX_i}_{n(-)}) = 1$ for $j \in \{1, \cdots, n-2 \}.$ 
\item If $ \bar{\ell}_{i}(\mathrm{G}_{n(-)}) \neq1$ 
then $\mathcal{N}(gx_i^{-1}) = \mathcal{N}(g) $ 
or $\mathcal{N}(gx_i^{\pm1}) = \mathcal{N}(g) - (n-1),$ 
$\bar{\ell}_{n}(\mathrm{GX_i}_{n(-)}) = \bar{\ell}_{n}(\mathrm{G}_{n(-)}).$
\end{enumerate}
  \label{lem25}
\end{lemma}
\begin{proof}
The proof is generalised from the arguments in \cite{MR3978542}. 
The elements $x_i \in \mathcal{G}$ where $i \in \{0, \cdots, n-2\}$ 
are the generators in $F_n$ in the finitely generating set 
with only two carets 
in each tree in the tree pair representation in Figure \ref{fig:generators}, 
i.e. $\mathcal{N}(x_i) = 2 (n-1) +1$ with the same target tree. 
The source tree of the unique tree pair representing $x_i$s has the second caret 
attached to leaves labelled $\{0, \cdots, n-1\}$, 
and has the branch $ik$ where $k \in \{0, \cdots, n-1\}$ in the source tree, respectively. 
The inverses $x_i^{-1}$s are represented by the tree pairs 
that reverse the source tree and the target tree in the pair.

The composition $gx_i$ can be represented 
by the multiplication of the unique tree pair representation of $g$ and $x_i.$
If $\bar{\ell}_{0}(\mathrm{G}_{n(-)}) = 1,$ 
the target tree of $g$ has a path with only one edge attached to the first leaf.
Hence the product add one more $n$-caret to 
both the source and the target tree of the tree pair representing $g.$ 
Alternatively, this is indicated by the branch $\ell_{00}(G_{n(+)}(x_i)) \mapsto \ell_{0i}(G_{n(-)}(x_i)).$ 
Since the number of carets in $gx_i$ is larger than the ones in $g,$ 
the equation in $(1)$ follows.
The branch $\bar{\ell}_{0}(GX_{i{n(-)}}) = 1$, 
since there is no extra $n$-carets added to the $i$th leaf on target tree of the product. 
%
This proves $(1),$ 
the arguments for the rest three cases are similar.
\end{proof}

\begin{corollary}
Let $g \in \mathcal{G}$ be any element that can be represented by 
a unique tree pair representation 
with at least three carets in each tree.
Then 
$ \mathcal{N}(g) < \mathcal{N}(gx_i^{+ m}) \leq \mathcal{N}(g) +m(n-1)-1$
when $\bar{\ell}_{0}(G_{n(-)}^{+1}) = 1$ 
and $\bar{\ell}_{0^m}(GX_{in(-)}^{+m}) = m+1;$ 
$ \mathcal{N}(g) < \mathcal{N}(gx_i^{-m}) \leq \mathcal{N}(g) +m(n-1)-1 $
when $\bar{\ell}_{i}(G_{n(-)}^{+1}) = 1.$
and $\bar{\ell}_{i0^{m-1}}(GX_{in(-)}^{-m})$ becomes $m+1.$
\label{cor37}
\end{corollary}

\begin{proof}
This is a direct consequence of the preceeding lemma.
\end{proof}

\begin{lemma}
Let $g \in \mathcal{G}$ be an element with the tree pair representation $(\mathrm{G}_{n(+)},\sigma, \mathrm{G}_{n(-)}).$ 
Let $\ell_{u}(\mathrm{G}_{n(+)})\mapsto \ell_v(\mathrm{G}_{n(-)})$ 
be a branch of $g$ and let $h$ be an element of $F_n$. 
Let $h' = (h)_{[v]}. $ 
Then $$\mathcal{N}(gh') = \mathcal{N}(g) + \mathcal{N}(h) - 1 .$$
\label{lem27}
\end{lemma}

\begin{proof}
The argument follows from the ones in the interpretation of the branches. 
According to Subsection \ref{branch} 
the unique tree pair representation of 
$h' = (h)_{[v]}$ has branches 
$\ell_{v\omega_s}{(\mathrm{T}_{n(+)}(h'))} \mapsto \ell_{v\omega_t}{(\mathrm{T}_{n(-)}(h'))},$
where $\omega_s$ and $\omega_t$ are the words 
corresponding to the paths in the source tree 
and the target tree of the tree pair representing $h$ respectively.
The branch $\ell_{u}{(g)} \mapsto \ell_{v}{(g)}$ in the product $gh'$ 
takes the source tree attached to vertex labelled $v$ in the tree pair representation of $h'$ back to the vertex labelled $u$ in the source tree pf the tree pair represents $g.$
Then we have branches $\ell_{u\omega_s}{(\mathrm{T}_{n(+)}(gh'))} \mapsto \ell_{v\omega_t}{(\mathrm{T}_{n(-)}(gh'))},$ in the tree pair representation of $gh'.$ 
We here give a more geometric interpretation of this process 
in Figure \ref{fig:fig3}.
\begin{figure}[htp]
\centerline{\includegraphics[width=5in]{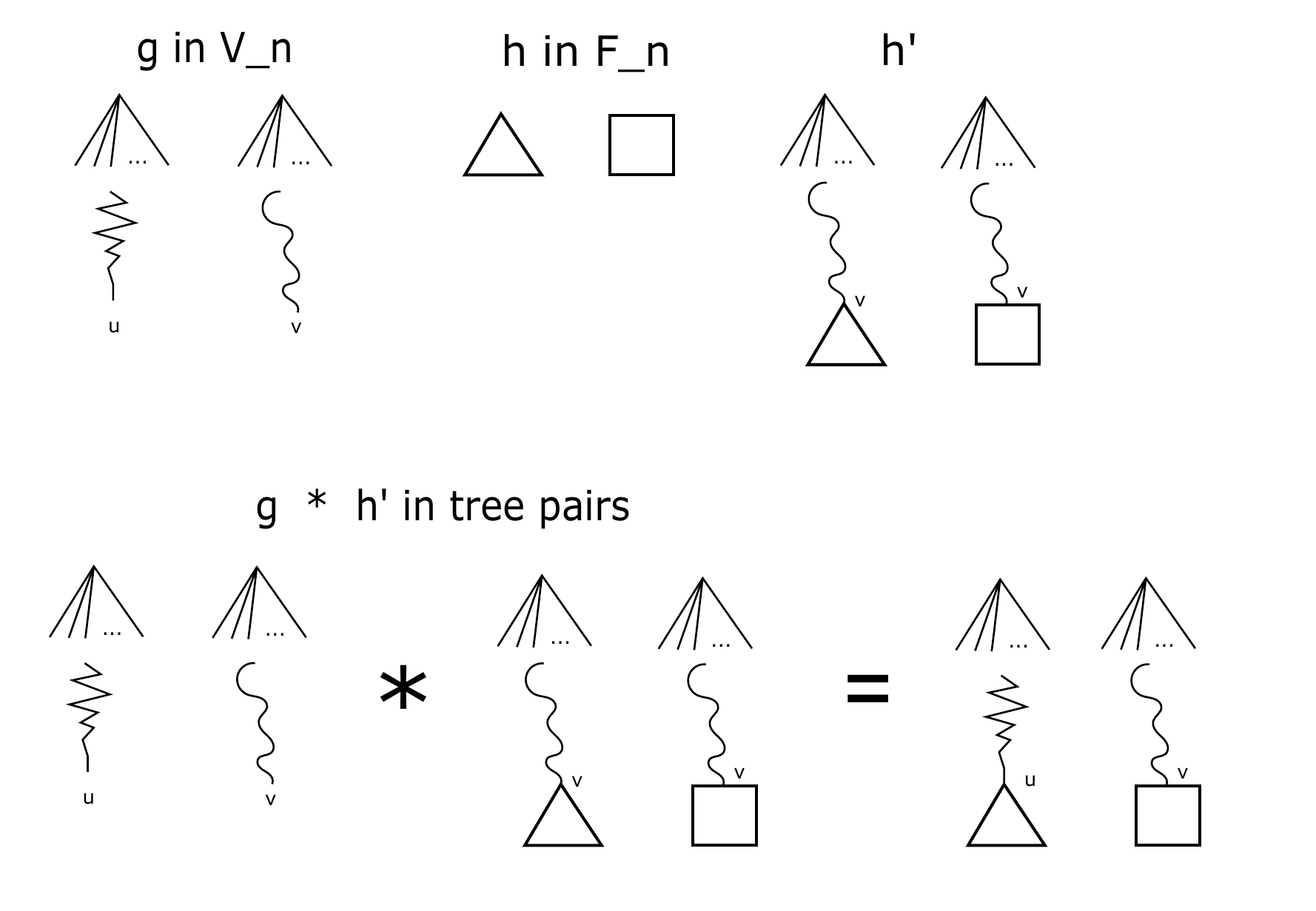}}
 \caption{$h' = (h)_{[v]}$ in tree pair representation \label{fig:fig3}}
\end{figure}
\end{proof}

\begin{proposition}[] 
There exist constants $\delta, D > 0$ and a positive integer $Q$ 
such that the following holds. 
Let $g \in \mathcal{G}$ be an element with $\mathcal{N}(g) \geq 2(n-1)+n = 3n-2 $. 
Then there is a path represented by $\omega$ 
of length at most $ D | g | $ in the Cayley graph $\Gamma = Cay(\mathcal{G}, X)$ 
which avoids a $\delta |g|$-neighbourhood of the identity 
and which has initial vertex $g$ and terminal vertex $g\omega$.
\label{prop3.9}
\end{proposition}

Similar to the strategy in the proof of \cite{MR3978542}, 
we want to find a path connecting two points on two geodesics staring from the same point, 
for instance, the identity $id,$
with some bounded distance and avoid a neighbouhood of $id.$
The construction works for $F_n$, $T_n$ and $V_n$. 

In the general case, 
we associate an element $g_k$ to some $k \in \mathbb{N}$ 
which is roughly the length of some group element $h$ in $F_n$,
and then find a path from $h$ to $g_k$ of which 
the length is linear to the word length of the element $h$. 
Finally, take two such elements so that the path between them 
which is linear to some of the positive numbers associate them 
as well as avoiding a ball at center $id.$ 

\begin{proof}[Proof of Proposition \ref{prop3.9}] 
Let $C$ and $c$ be the constants from Proposition \ref{prop1}. 
Pick an element $h$ randomly,
and associate to the word length of $h$ a group element $g_{| h | }$ 
so that the path $\omega$ between these two group elements in the Cayley graph is linear 
to the word length $| h|$ 
with respect to the standard finite generating set.

For some group element $h$ whose number of carets is at least three 
and hence the number of leaves is at least $2(n-1)+n,$
denote by $(\mathrm{H}_{n(+)}, \sigma_h, \mathrm{H}_{n(-)})$ the tree pair representation of $h.$ 
We want the element to be supported on some interval $ [0, a_{|h|}).$ 
Hence we concatenate the subpath 
$\omega_1 = x_{0}^{2}x_{n-1}^{-1}x_0^{-1}$ to obtain $h_1 = h\omega_1$ if $l_{0}(h) =1$, 
otherwise we take $\omega_1 $ to be an empty word.

\begin{figure}[h]
\centerline{\includegraphics[width=5in]{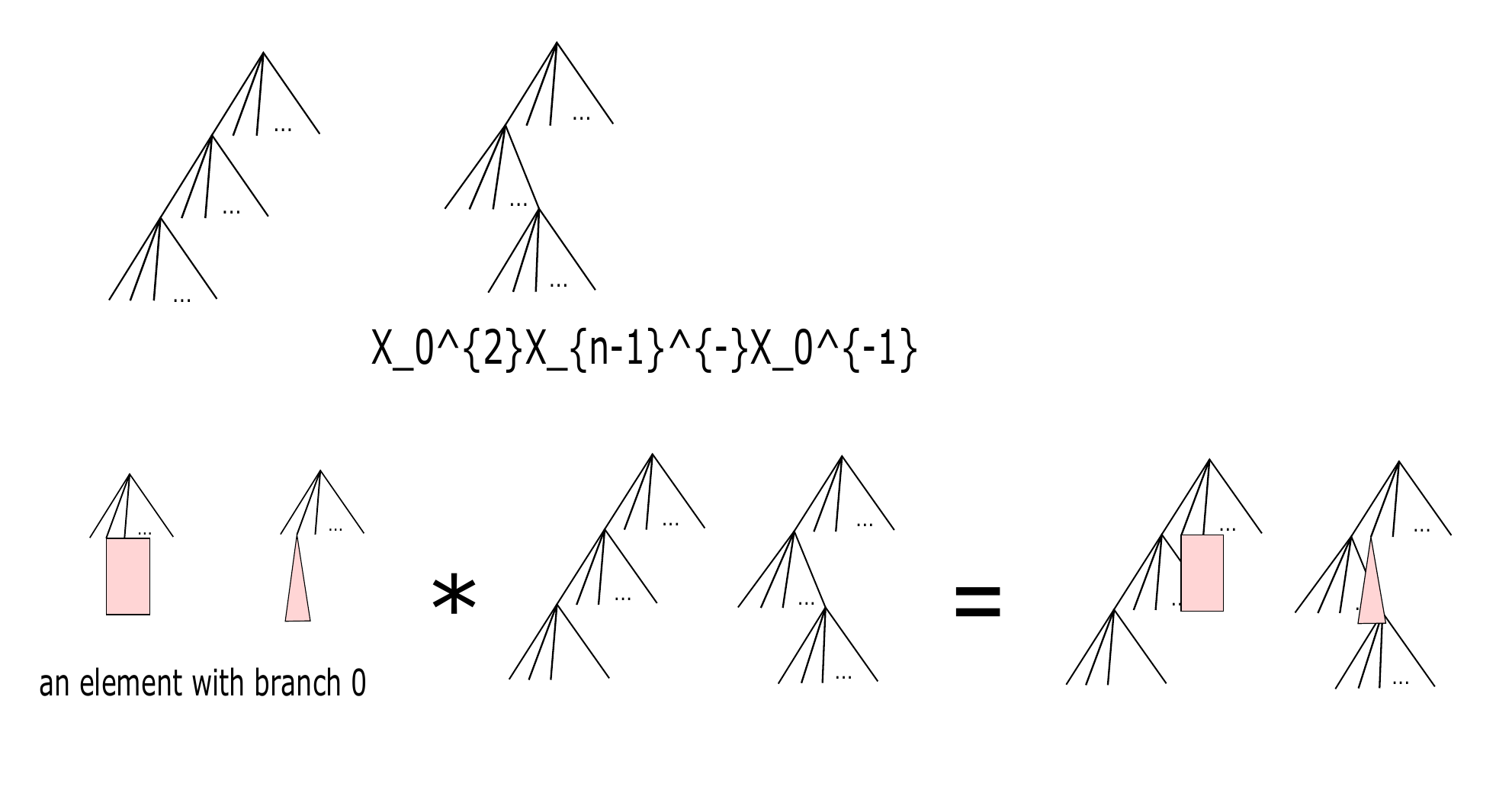}}
\caption{Adding the first subpath \label{fig4}}
\end{figure}

Next, we want to add on two subpaths to ensure that the newly created word represent an element in $F_n.$
We concatenate the subpath $\omega_2$ and $\omega_3$ to $h_1$ 
and the construction of the two subpaths are as follows:
The subpath $\omega_2 = x_0^{-M\mathcal{N}(h_1)} x_{n-1}x_0^{M\mathcal{N}(h_1)}$ 
which is an element in $F_n$ 
and we take $M \geq 10\frac{C}{c(n-1)}.$  

The subpath $\omega_3$ is a word representing 
an cyclic element $h_{\omega_3}$ in $T_n < V_n$ 
such that the tree pair representation of the element $h_{\omega_3}$ 
has the tree pair 
with the source tree being the target tree of the tree pair of $h$ 
and the target tree being the source tree of the tree pair of $h,$
i.e. $(\mathrm{H}_{n(-)},\sigma_{\omega_3},\mathrm{H}_{n(+)})$
$\sigma_{\omega_3}$ 
permutes the leaves from $\mathrm{H}_{n(-)}$ to $\mathrm{H}_{n(+)}$ in the a cyclic order 
and with the permutation $\sigma_{\omega_3}(\sigma_(h_1))(\ell_0(h_{\omega_3})) =\ell_0(h_{\omega_3}),$ 
in other words, 
$h_{\omega_3}$ has the branch $\ell_{(\sigma_{\omega_3}(0))} \mapsto \ell_{0}.$ 
Since $h_{\omega_3} \in T_n,$ its word length is bounded linearly by its number of carets and hence the number of leaves,
i.e. $| h_{\omega_3} | \leq C \mathcal{N}(h_{\omega_3}).$  

By concatenating the paths, 
we have $h_{3} = h\omega_1\omega_2\omega_3.$ 
From the construction, 
we have that 
$\mathcal{N}(h_1) \leq \mathcal{N}(h_1) +(m-1)(n-1) -1 \leq \mathcal{N}(h_2) \leq \mathcal{N}(h_3) \leq 2\mathcal{N}(h_1) +(m-1)(n-1) -2.$
First two in equalities are from Lemma \ref{lem27}, 
the last ones is from considering the branches with prefix having $0$s and Corollary \ref{cor37}.

The subpath $\omega_4$ represents an element in $F_n$ 
with the construction $x_0^{Q| h |} x_{n-1}^{-1} x_0^{- Q| h | +1}$
where we take $Q \geq 10\frac{M}{c^2(n-1)}.$
This subpath ensures 
$\mathcal{N}(h_3) \leq \mathcal{N}(h_3\omega_4)$ 
and commutes with $ h_3.$
This is because $h_3$ is supported on $[0, \ell_{0}(h_3) )$ 
while $\omega_4$ from the definition fixes the same interval. 
We take the last subpath $\omega_5$ to be the word representing $h_{3}^{-1}$, 
where $h_{3} = h\omega_1\omega_2\omega_3.$ 

Now, we can compute the upper bound of the length of the path $\omega$ as follows:
\begin{align*}
& \lVert \omega \rVert  \leq  \lVert h_3\rVert + \lVert \omega_4\rVert + \lVert \omega_5\rVert   
= 2\lVert h_3\rVert + \lVert \omega_4\rVert   \\
 &\leq 2(2M\mathcal{N}(h_1)  +2\mathcal{N}(h_1)) + (2Q\mathcal{N}(h_1) +2) \\
&\leq 4M\mathcal{N}(h_1) +(2Q+2)\mathcal{N}(h_1) + 2  \\
&\leq (4M+2Q+2) \mathcal{N}(h_1)  +2\\
 & = \Big(\frac{4M+2Q+2}{c} \Big) | h | +2
\end{align*}

%
We could take $\delta =\frac{c}{8M}$ since $\lVert \omega \rVert  \leq 4M\mathcal{N}(h) $ 
and $D = \Big(\frac{4M+2Q+4}{c} \Big) | h |,$ 
so that $\delta | g_{| h|}| \leq\lVert g_{| h|}\omega' \rVert \leq D| g_{| h|}|$ 
which proves the linearity of the length of the path from $h$ to $ g_{| h |}.$
\end{proof}

\begin{lemma}
For a prefix $\omega'$ of $\omega$, $| h \omega' | \geq \delta | h |.$ 
\end{lemma}
\begin{proof}
By considering the changes in the number of carets, 
there is no reduction on the number of carets while adding on the first subpath, 
since when adding on the first subpath $\omega_1,$ 
if $\bar{\ell}_0(H_{-}) =1$, $\mathcal{N}(h\omega_1) > \mathcal{N}(h) $ by Corollary \ref{cor37}, 
otherwise we add an empty word and the length is unchanged. 
For $\omega_2,$ Lemma \ref{lem25} ensures the increase in the number of carets 
and hence the number of word length. 
The tree pair representation of $h_{\omega_3}$ have the same number of carets as in $h_1$ in the $\textbf{pcq}$ form. 
Besides $\omega_3$ has a cyclic permutation, 
hence by \cite{Sheng:2018aa} whose word length is linear to the number of its carets 
and together with Lemma \ref{lem27}, 
we have $\mathcal{N}(h_1\omega_2) \leq \mathcal{N}(h_1\omega_2\omega_3) =\mathcal{N}(h_3).$
Since the tree pair representation of $h_{\omega_3}$ 
has the branch $\ell_{(\sigma_{\omega_3}(0))} \mapsto \ell_{0},$
$h_3$ is supported on some $ [0, a_n).$
The path $h_2$ after adding on $\omega_2$ ends with powers of $x_0^{-1},$ 
we choose $\omega_3$ to be the shortest word representing the element 
$h_{\omega_3}$ with the tree pair symmetric to the tree pair representing $h.$
Since no reduction occurs when concatenating powers of $x_0$s to $h,$
there are also no reductions occurring when concatenating $\omega_3$ to $x_0^{-1}$s 
which is just reflecting the picture of the operating of the product of the tree pairs.

For the subpath represented by word $\omega_4,$ 
the argument is the same as the ones for the case when adding the subpath representing $\omega_2.$ 
As for adding on the subpath $\omega_5,$ 
it is just the inverse of path $h\omega_1\omega_2\omega_3.$
\end{proof}

With the preceding Proposition, we could prove Theorem \ref{thm1}.
\begin{proof}[Proof of Theorem \ref{thm1}] 
From Proposition \ref{prop3.9}, 
we could pick two elements $g_1$ and $g_2$ randomly,
for each $g_i$ where $i \in\{1, 2\}$,
there is a path $\omega_i$ avoiding some ball centered at identity 
with radius $D(| g_i|)$ 
from element $g_i$ to $x_0^{Q | g_i |} x_{n-1}^{-1} x_0^{- Q | g_i | +1}.$ 
Assume without loss of generality, we have $| g_1 | \leq | g_2 |$ 
and we add a path $$p = x_0^{Q |g_1 |-1} x_{n-1} x_{0}^{- Q | g_1-g_2 | +1} x_{n-1}^{-1} x_0^{- Q | g_2 | +1}.$$
Then by the previous lemma and the estimate from of the word length,
we have 
$$  \delta \min\{| g_1|, | g_2 | \}  < cQ \rVert g_1 \lVert \leq | x_0^{Q | g_1 |} x_{n-1}^{-1} x_0^{- Q | g_1\omega'' |} |.$$
\end{proof}
For some undistorted subgroups of Thompson's groups, 
the method of finding out the comparison paths for the divergence functions is similar, 
since the estimate of the word length replies only on the subgroups $F_n$. 
%
%
%

Alternatively, we could rely on the quasi-isometric embeddings.
Since, the Brown-Thompson groups $F_n$, $T_n$ can be embedded quasi-isometrically 
into Thompson's groups $F$, $T,$ respectively, 
we could consider the quasi-isometric embedding $\phi^{\ast}$ from $T_n$ to $T$ defined in 
\cite{Sheng:2018aa}, 
and pick the shortest path 
between the image of two points outside the ball centred at the identity 
with radius of the element with larger word length in $T,$ 
then the paths that are quasi-isometric to the preimage of this linear path \cite{Sheng:2018aa} 
taken back by $\phi^{\ast}$ may also be an estimate.



\section{Divergence functions for the braided Thompson's group}

In \cite{MR3978542}, Golan and Sapir posed the question: 
if the divergence property of the braided version of Thompson's groups 
can also be found using the same method. 
In this section, 
we would like to investigate this generalised version of Thompson's group, 
the braided Thompson groups.
The origin of these groups can be traced back to Brin \cite{MR2364825} 
and Dehornoy \cite{MR2105432} of which the metric properties are being investigated 
in \cite{MR2384840, MR2514382} 
where the estimate of the word length is slightly different.

\subsection{The braided Thompson groups}

The braided version of the Thompson groups can be regarded 
as some ``Artinification" construction of the original Thompson's groups.
Intuitively, 
when considering the tree pair representation of the elements of Thompson's groups, 
instead of just permuting the labelings of the leaves, 
we add ``braids" to the leaves between the pair.
One replaces the permutations by the braid relations. 
The group $BV,$ the larger group $\widehat{BV}$ and the braided analogues of $F,$ 
namely, $BF$ and $\widehat{BF}$ will be considered in the rest of this paper.

%
%
We recall dyadically subdivided intervals defined in the previous section. 
For two dyadically subdivided intervals 
with the same number of subintervals $I$ and $J,$
we associate them to an element in the braid group $B_m$ 
where $m$ is the number of the subintervals. 
In order to navigate the map between subintervals in $I$ and $J,$ 
we identify each subinterval with the thickened strands or a elastic band in the braids.
\begin{figure}[h]
\centerline{\includegraphics[width=2.25in]{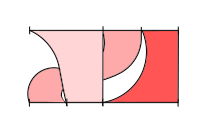}}
 \caption{ \label{fig:fig10}}
\end{figure}
Figure \ref{fig:fig10} illustrates one of the group elements 
of the braided Thompson group $BV.$ 
Adding and subtracting the dyadic breakpoints 
to the subintervals to the corresponding strands do not change the map.
When there are two such maps, 
i.e. a map from $I_1$ to $J_1$ and a map from $I_2$ to $J_2$, 
where $I_1$, $J_1$, $I_2$, $J_2$ are dyadically subdivided unit intervals, 
and $I_1$, $J_1$ are the intervals with the same number of subdivisions, 
$I_2$, $J_2$ are the intervals with the same number of subdivisions, respectively. 
We compose $J_1$ and $I_2$ as composing two piecewise-linear functions 
and then extend the composition to $I_1$ and $J_2$ as in Figure \ref{fig:fig11},
the resulted map consists again a pair of dyadically subdivided unit intervals 
associated to an element in the braid group. 
\begin{figure}[h]
\centerline{\includegraphics[width=3.75in]{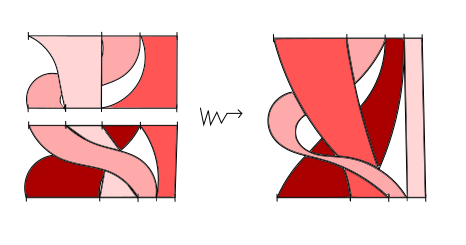}}
 \caption{ \label{fig:fig11}}
\end{figure}
\begin{definition}[$BV$] 
The braided Thompson group $BV$ is 
the group of the piecewise-linear maps 
from the set of the dyadically subdivided intervals to itself 
associating pairs of intervals with intervals having the same number of breakpoints, 
such that the subintervals in the source unit interval are associated 
to those in the target as follows,
\begin{itemize}
\item all subintervals in the source and the corresponding ones in the target 
are bijectively connected to each other 
by gluing a rectangular flat elastic band (or a ribbon) to each pair of them 
such that the top of the band is glued to the source subintervals 
and the bottom is glued to the the target subintervals;
\item the band is not allowed to flip or twist, but can be stretched;
\item the set of the bands can be identified 
with the thickened strands of a braid with $m$ strands, $m \in \mathbb{N}$ 
and can be represented by an element $b$ in the braid group $B_{m}$ 
where $m$ is the number of the subintervals in both 
the source and the target unit intervals.
\end{itemize}
The group operation is defined by:
First, compositing the target interval in the 
band-connected interval pair 
representing the former element 
and the source interval 
in the pair 
representing the latter element. 
Then by extending the composition through the bands 
back to the source interval in the pair representing the former element 
and also to the target interval in the pair of the latter one.  
\label{def41}
\end{definition}
This group can also be regarded as 
the extension of Thompson's group $V$ by the ``stable braid groups " with trivial center 
which includes all braid groups \cite{MR2952772} by $V$. 
Similar to the case in $F$, $T$, $V$ and the Brown-Thompson groups, 
there is a more combinatorial interpretation using the tree pair representation 
and the braid diagrams for the braided Thompson groups.

These are called tree-braid-tree diagram 
or the braided version of the $\textbf{pcq}$ factorisation 
where the middle component $\textbf{c}$ comes from 
the elements of the braid groups instead of the symmetric groups.
Since for $F$, $T$, $V$ as well as for $F_n$, $T_n$ and $V_n,$ 
we have a symbol with three tuples 
indicating the tree pair representation for a group element,
analogously, 
we have a similar symbol $(\mathrm{T}_{+} , \beta, \mathrm{T}_{-} )$ for elements in $BV$ 
where $\beta$ is a braid induced by an element in the braid group $B_{m+1}$ with $m$ strand.

\begin{figure}[h]
\centerline{\includegraphics[width=5in]{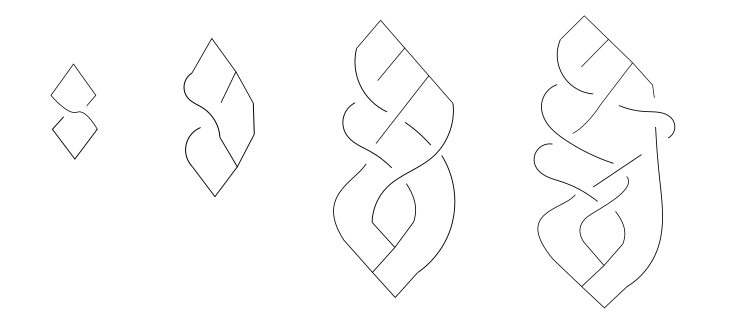}}
 \caption{ $\tau_1$, $\sigma_1$, $\phi \in BV$, $\psi \in BF.$ \label{fig:fig5}}
\end{figure}

%
Following \cite{MR3545879}, 
we allow the tree-braid-tree diagrams to change to each other 
through the following equivalent movements in Figure \ref{fig:fig13} 
which form an equivalence class 
and hence define the composition of the diagrams.
\begin{figure}[h]
\centerline{\includegraphics[width=4.5in]{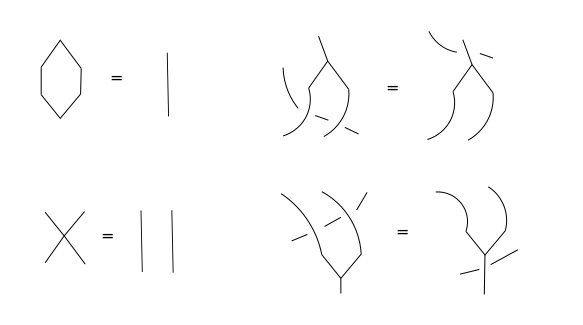}}
\vspace*{8pt}
 \caption{ \label{fig:fig13}}
\end{figure}
When we want to compose two tree-braid-tree diagrams, 
we lay the two diagrams vertically,
stacking the first diagram pair on top of the second diagram pair
and merge the target tree of the top diagram 
and the source tree of the bottom tree 
using the relations on the left side on Figure \ref{fig:fig13}, 
then using the relations on the right side to drag the carets up or down 
to obtain a new tree-braid-tree diagram.  
This operation coincides with the multiplication of the group elements.

Similarly, we defined the braided version for the group $F$.
\begin{definition}[$BF$] 
The braided Thompson group $BF$ is 
the group of the orientation-preserving piecewise-linear isotopies 
from the set of the dyadically subdivided intervals to itself 
associating pairs of intervals with intervals having the same number of breakpoints,
such that the subintervals in the source unit interval are associated 
to the ones in the target as follows,
\begin{itemize}
\item all subintervals in the source and the corresponding ones in the target 
are connected together 
by gluing a rectangular flat elastic band to each pair of them
such that the top of the band is glued to the source subintervals 
and the bottom is glued to the corresponding target subintervals;
\item the band is not allowed to flip or twist, but can be stretched;
\item the set of the bands can be identified 
with a thickened pure braid with $m$ strands, $m \in \mathbb{N}$ 
and can be represented by an element $p$ in the pure braid group $PB_{m}.$
In addition, $m$ is the number of the subintervals in both the source and the target unit intervals.
\end{itemize}
The group operation is defined in the same manner 
as the one in the Definition \ref{def41}.
\label{def42}
\end{definition}

Both $BF$ and $BV$ are finitely presented \cite{MR2384840} 
and $BV$ has a finite generating set $x_0$, $x_1$, $\sigma_1$ and $\tau_1$ 
where $x_0$ and $x_1$ come from the standard generating set of $F$ 
and the other two are induced by the braid relations 
and are illustrated on the left half in Figure \ref{fig:fig5}.
These two groups also have series of infinitely generating sets $x_i^{\pm1}$, $\sigma_i^{\pm1}$,$\tau_i^{\pm1}$ 
where $x_i^{\pm1}$ are the infinite order generators from the original Thompson's groups 
and $\sigma_i^{\pm1}$,$\tau_i^{\pm1}$ are the infinite sequence of generators induced from the braid groups 
\cite{MR2514382, MR3781416} 
which can be represented by tree-braid-tree diagrams with a pair of all right trees 
and a positive crossing at the strands on $(i-1)$th and $i$th leaves for the diagram pair representing $\tau_i^{\pm},$ 
a positive crossing at the strands on $(i-2)$th and $(i-1)$th leaves for diagram pair $\sigma_i^{\pm1}.$

Similar as in the non-braided version of Thompson's groups, 
a set of the tree-braid-tree diagrams 
under some equivalence relation 
have a unique reduced diagram pair 
which corresponds to a standard form.
A more precise description of the equivalent relation of the diagrams 
and the minimal size representatives can be referred to \cite{MR1806724}.
Since the estimate on the word length later on will depend largely on 
this standard form, 
we give a brief introduction on this combinatorial description below.

\begin{definition}[The braided version of $\textbf{pcq}$ form]
Let the reduced labelled tree pair $(\mathrm{T}_{(+)},\beta, \mathrm{T}_{(-)} )$ 
represent an element $g \in BV$ 
and $\mathrm{T}_{(+)}$ and $\mathrm{T}_{(-)}$ each have $i$ carets. 
Let $\mathrm{R}$ be the all right tree with $i$ carets. 
We write $g$ as a product $\textbf{p}\beta\textbf{q}$, where:
\begin{enumerate}
\item $\textbf{p}$,  a positive word in the infinite generating set of $F$ 
and with the form 
$\textbf{p} = x_{i_1}^{r_1}x_{i_2}^{r_2} \cdots x_{i_y}^{r_y} $ 
where $i_1 < i_2 \cdots < i_y $ 
with tree pair $(\mathrm{T}_{(+)}, \rm{id}, \mathrm{R}).$
\item $\beta$ part is induced by an element in the braided group of $m$-strands, 
which is represented by a $(\mathrm{R},\beta,\mathrm{R}),$ 
and is a word in the infinite generating sets $\{\sigma_i\}$ and $\{\tau_i\}$ 
which are induced by the elements in $B_{m}$ where 
$m = \max\{ \mathcal{N}(\mathrm{T}_{(+)}), \mathcal{N}(\mathrm{T}_{(-)}) \} .$ 
\item $\textbf{q}$, a negative word, 
is the in the form $x_{j_z}^{-s_z} \cdots x_{j_2}^{-s_2}x_{j_1}^{-s_1}$ 
where $j_1< j_2 <\cdots < j_z$ for an element in $F$ 
represented by $(\mathrm{R},\rm{id}, \mathrm{T}_{(-)}).$ 
\end{enumerate} 
We call such a product $\textbf{p}\beta\textbf{q}$, a $\textbf{pcq}$ factorization for the braided Thompson groups. 
\end{definition}
Note that a word with the form 
$x_{i_1}^{r_1}x_{i_2}^{r_2} \cdots x_{i_y}^{r_y} \mathrm{id} x_{j_z}^{-s_z} \cdots x_{j_2}^{-s_2}x_{j_1}^{-s_1} $ in $BV,$
where $i_1 < i_2 \cdots < i_y  \neq j_z >\cdots > j_2 > j_1 $, $r_k$'s and $s_l$'s are all positive 
and $1 \leq j \leq i(n-1)+1$, 
is a word 
with respect to the infinite generating set of $F.$

%
%
%
In \cite{MR2384840}, there is a silmilar notion, \textit{block}. 
By applying reductions to the words in the $\textbf{pcq}$ form 
we may naturally obtain a block.

This is another standard form of the group elements 
when taking the word $\beta$ to be the right-greedy form \cite{KD08}.
\begin{definition}[Block\cite{MR2384840}]
A word in the infinite generators $x_i^{\pm1}$, $\sigma_i^{\pm1}$, $\tau_i^{\pm1}$ is called a block,
if it is of the form $\omega_1\omega_2\omega_3^{-1}$ where 
\begin{enumerate}
\item $\omega_1$ is a positive word as the $\textbf{p}$ part in the $\textbf{pcq}$ representation.
\item Let $N = \max{\{ N(\omega_1), N(\omega_3)\}},$ 
then there exists an integer $m \geq N +1 $ such that $\omega_2$ is a word 
in the generators $\{ \sigma_1, \sigma_2, \cdots, \sigma_{m-2}, \tau_{m-1} \}.$
\item  $\omega_3^{-1}$ is a negative word as the $\textbf{q}$ part in the $\textbf{pcq}$ representation.
\end{enumerate}
\end{definition}

Notice that when the $\textbf{c}$ part of the $\textbf{pcq}$ form of some group element 
is a word in the generators induced from some braid group 
then the word can easily be transformed to a block.
\begin{lemma}
Every group element in the \textbf{pcq} form 
can be transformed to a block and transform back 
by finitely many steps of moving around letters.
\end{lemma}
\begin{proof}
The $\textbf{p}$ and the $\textbf{q}$ are elements in $F$,
they are equivalent to the first block and the third block in the block form 
and can possibly be reduced to shorter words.
The middle part of the $\textbf{pcq}$ form denoted by $\beta$, 
is a word in $\sigma_i^{\pm1}$,$\tau_i^{\pm1}$ 
induced by elements in $B_{m}$ for some positive integer $m$, 
hence, more precisely, 
a word in $\{\tau_1, \cdots, \tau_{m-1}\}$ and $\{\sigma_1, \cdots, \sigma_m\}.$
Leave the letters in $\{\sigma_1, \cdots, \sigma_m\}$ where they are 
and rewrite the letters from $\{ \tau_2, \cdots, \tau_{m-1} \}$
by reducing the subsuffix of the generators using the relation $D_2$ and $D_1$ 
from \cite{MR2384840} to move around the generators $x_{i}.$ 
Since $\beta$ has only finite letters, 
the process takes only finitely many steps.
\end{proof}

In the original definition \cite{MR2364825}, 
there is a ``larger" version of the braided Thompson group 
which contains the group $BV$ 
and also sits inside $BV$ 
which has the last strand unbraided and we define this group formally below:
\begin{definition}[$\widehat{BV}$]
The braided Thompson group $\widehat{BV}$ is defined as
the group of the piecewise-linear maps 
from the set of the infinite dyadically subdivided intervals to itself 
such that the first $k$ subintervals in the source interval 
are associated to the first $k$ subintervals in the target 
as an element in $BV$ in the description of Definition \ref{def41}, 
where $k \in \mathbb{N}.$ 
The rest of the subintervals in the source and the target
are associated to each other one by one from the left in order. 
\end{definition}
Similarly, we could the define the larger version of the braided Thompson groups containing $F.$
\begin{definition}[$\widehat{BF}$]
The braided Thompson group $\widehat{BF}$ is 
the group of the orientation-preserving piecewise-linear isotopies 
from the set of the infinite dyadically subdivided intervals to itself 
such that the first $k$ subintervals in the source interval 
are associated to the first $k$ subintervals in the target 
as an element in $BF$ in the description of Definition \ref{def42},
where $k \in \mathbb{N}.$  
The rest of the subintervals in the source and the target
are associated to each other one by one from the left in order.
\end{definition}

Since the braided version of Thompson's groups involve both Thompson's groups 
and braid groups, 
we expect the estimate of the word length 
to depend on the estimate of the word length of both of Thompson's groups 
and braid groups. 

\subsection{Divergence property of the braided Thompson groups $BF$, $BV$} %

By considering the group elements of the braided Thompson groups as tree-braid-tree diagrams,
we estimate the word length from a purely combinatorial aspect.

For Thompson's group $F$ and the generalisations $F_n$ 
without the torsion elements,
the word length is proportional to the number of carets in the tree pair representation, 
whereas for the braid groups, 
the word length with respect to the standard generators is 
closely related to the number of the crossings. 
Here we denote by $\mathcal{K}(\cdot)$, the number of crossing of the tree-braid-tree diagram representing some group element.
In \cite{MR2514382}, Burillo et. al. have considered the Garside elements of the braid groups 
inside the $BV$ in order to give the following estimate. 
%

\begin{lemma}[\cite{MR2514382}]
For an element $\omega$ in $BV$ 
which has a reduced tree-pair-tree diagram representative with $n$ leaves and $k$ total crossings, 
and with the maximum number of crossings of a pair of strands is $s$, 
then there exist constants $C_1$, $C_2$, $C_3$, $C_4$ such that the length $| \omega | $ satisfy the following:
$$C_1\max\{n, \sqrt[3]{k}\} \leq C_2 \max\{n, s\} \leq | \omega | \leq C_3 (n+nk) \leq C_4 (n+n^3s).$$
\label{lem2}
\end{lemma}
Garside elements which provide the largest number of crossings 
for the tree-braid-tree representation of the braid elements in Thompson's group are the ones 
that we want to avoid in the estimate of the word length.
 \begin{lemma}{}
For $g \in  BV$ being represented 
by some reduced tree pair representation $(\mathrm{G}_{(+)},\sigma_g, \mathrm{G}_{(-)})$ 
with at least three carets in the each tree, 
and $gx_0$ be represented by $(\mathrm{GX_0}_{(+)},\sigma_{gx_0}, \mathrm{GX_0}_{(-)}),$ 
let $\bar{\ell}_{o}(\cdot)$ denote length of the branch of a tree, 
then we have 
$$\mathcal{N}(g) - 1 \leq  \mathcal{N}(gx_{0}^{\pm1}) \leq \mathcal{N}(g) + 1 $$
In addition,
\begin{enumerate}
\item If $\bar{\ell}_{0}(\mathrm{G}_{(-)}) = 1,$ 
then $\mathcal{N}(gx_0) = \mathcal{N}(g) +1$ 
and $\bar{\ell}_{0}( \mathrm{GX_0}_{(-)}) = 1.$
The number of the crossing is at most twice the original number.
\item If $ \bar{\ell}_{0}(\mathrm{G}_{(-)}) \neq 1,$ 
then $\mathcal{N}(gx_0) = \mathcal{N}(g) $ 
or $\mathcal{N}(gx_0) = \mathcal{N}(g)-1.$
Moreover, 
$\mathcal{N}(gx_0) = \mathcal{N}(g)$ 
only when $1$ and $01$ are strict prefixes 
are prefixes of the target tree $G_{(-)}$ in the pair representing element $g.$ 
No changes in the number of crossings.
\item  If $ \bar{\ell}_{1}(\mathrm{G}_{(-)}) = 1,$ 
then $\mathcal{N}(gx_0^{-1}) = \mathcal{N}(g) +1$ 
and  $\bar{\ell}_{1}( \mathrm{GX_0^{-1}}_{(-)}) = 1.$
The number of the crossing is at most twice the original number.
\item  If $ \bar{\ell}_{1}(\mathrm{G}_{(-)}) \neq 1,$ 
then $\mathcal{N}(gx_0^{-1}) = \mathcal{N}(g) $ 
or $\mathcal{N}(gx_0^{-1}) = \mathcal{N}(g)-1.$
Moreover, 
$\mathcal{N}(gx_0) = \mathcal{N}(g)$ 
only when $10$ and $11$ are strict prefixes 
are prefixes of the target tree $G_{(-)}$ in the pair representing element $g.$ 
No changes in the number of crossings.
\end{enumerate}
\label{lem39}
\end{lemma}
\begin{proof} The proof is a generalised argument of the proof of Lemma \ref{lem25} 
and the estimate of the number of crossing is purely by graphical computation.
Take $g$ such that $\bar{\ell}_{0}(\mathrm{G}_{(-)}) = 1,$ 
the first statement follows from the changes in the carets.
The change is the number of crossings depending changes in the strands 
under the extra carets added to the product.
Since we have one more caret in each tree in the tree-braid tree diagram of $gx_0,$
we then have at most twice the number of the original number of crossings. 
The rest three cases are by similar arguments. 
\end{proof}
We hence know that the generators in the braided Thompson groups 
that are coming from the torsion free part of Thompson's groups 
have affects mainly on the number of the carets of the tree-braid-tree diagrams.

\begin{lemma}
For $g \in  BV$ being represented by some reduced tree pair representation 
$(\mathrm{G}_{(+)},\sigma_g, \mathrm{G}_{(-)})$ 
with at least three carets in the each tree, 
let $\ell_{\gamma}(\cdot)$ denote the branch of an element 
such that $\gamma$ is a branch of the source tree in the reduced tree pair representing $g$ as before, 
$\bar{\ell}_{\gamma}(\cdot)$ denote the length of the branch of some tree 
and 
let $\mathcal{K}(\cdot)$ be the number of crossings in the reduced tree pair representation 
(or tree-braid-tree diagrams)
of an group element. 
$$\mathcal{N}(g)  \leq  \mathcal{N}(g\sigma_1^{\pm1}) \leq \mathcal{N}(g) + 1 $$
$$ \mathcal{N}(g\tau_1^{\pm1}) = \mathcal{N}(g)  $$
In addition,
\begin{enumerate}
\item If $\bar{\ell}_{0}(G_{(-)}) = 1,$ 
then $\bar{\ell}_{1}(G_{(-)}) \neq 1,$
$\mathcal{N}(g\sigma_1^{\pm1}) = \mathcal{N}(g)$ 
and 
\begin{enumerate}
\item $\bar{\ell}_{0}( \mathrm{GX_0}_{(-)}) = 1 $ when $\ell_{10}( \mathrm{GX_0}_{(-)})$ is a branch, 
$\mathcal{K}(g\sigma_1^{\pm1}) \leq \mathcal{K}(g)+1;$ 
\item $\bar{\ell}_{0}( \mathrm{GX_0}_{(-)}) \neq 1 $ when $\ell_{10}( \mathrm{GX_0}_{(-)})$ is not a branch,
$\mathcal{K}(g\sigma_1^{\pm1}) \leq 2 (\mathcal{K}(g)-2);$ 
\end{enumerate}
 \item If $\bar{\ell}_{1}(G_{(-)}) = 1,$ 
then $\bar{\ell}_{0}(G_{(-)}) \neq 1,$
$\mathcal{N}(g\sigma_1^{\pm1}) = \mathcal{N}(g)+1,$ 
and $\bar{\ell}_{0}( \mathrm{GX_0}_{(-)}) \neq 1.$
Moreover, $\mathcal{K}(g\sigma_1^{\pm1}) \leq \mathcal{K}(g)+2;$ 
\item The crossing number of $g\tau_1^{\pm},$ $\mathcal{K}(g\tau_1^{\pm}) \leq \frac{\mathcal{K}(g)^2}{4}.$
%
%


                  

                   
\end{enumerate}
\end{lemma}
\begin{proof}
When $\bar{\ell}_0(G_{(-)}) =1,$ 
$\bar{\ell}_{1}(G_{(-)}) \neq 1$ follows from the assumption, 
$\sigma_1^{\pm1}$ takes $0$ to $10,$
takes $10$ to $0$ and $11*$ to $11*.$
Then number of carets is hence unchanged through the computation.
Since there is one crossing between the first and second strand in $\sigma_1^{\pm1},$ 
i.e. strand on leaves labelled $0$ and $10,$ 
for the number of crossings of $g\sigma_1^{\pm1},$ 
we only need to consider how these two strands are interacting with other strands.
For $(a),$ the product does not add extra caret 
and hence the number of crossings just adds by one from $\sigma_1^{\pm}.$
Similar arguments work for the case $(b).$
For $(2)$ and $(3),$ we do the similar analyse.
\end{proof}


\begin{example}
Some elementary cases of the multiplication of the generators 
induced from braid relations $\tau_1$, $\sigma_1$ 
and the Thompson group generators $x_0$ 
are illustrated in Figure \ref{fig6} and Figure \ref{fig7}.

\begin{figure}[htp]
\centering
\subfloat[data The products of $\tau_1$ and $x_0$, $x_1$ ]{
\includegraphics[width=3in]{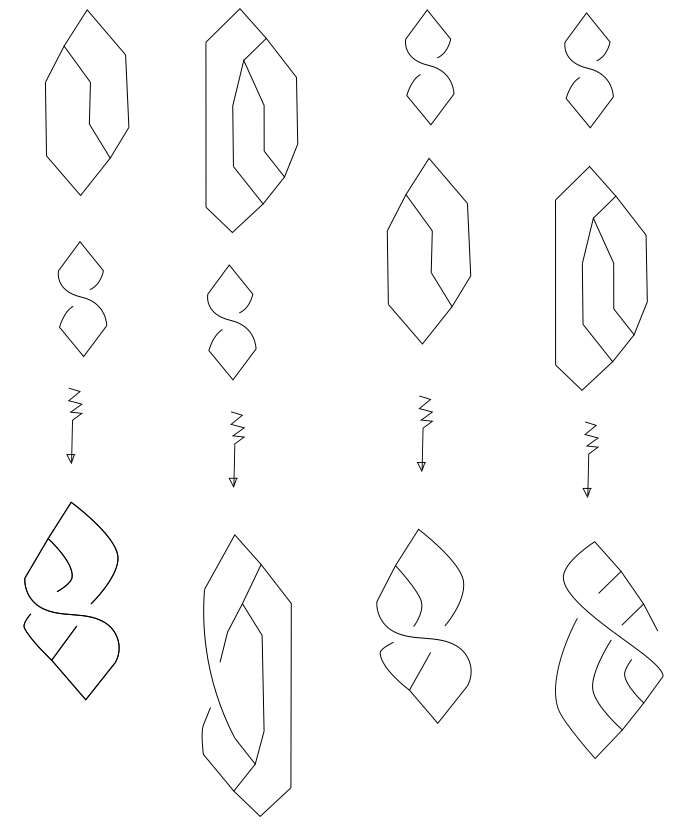}
\label{fig6}
}
\subfloat[data $\sigma_1$ and $x_0$, $x_1$]{
\includegraphics[width=2.5in]{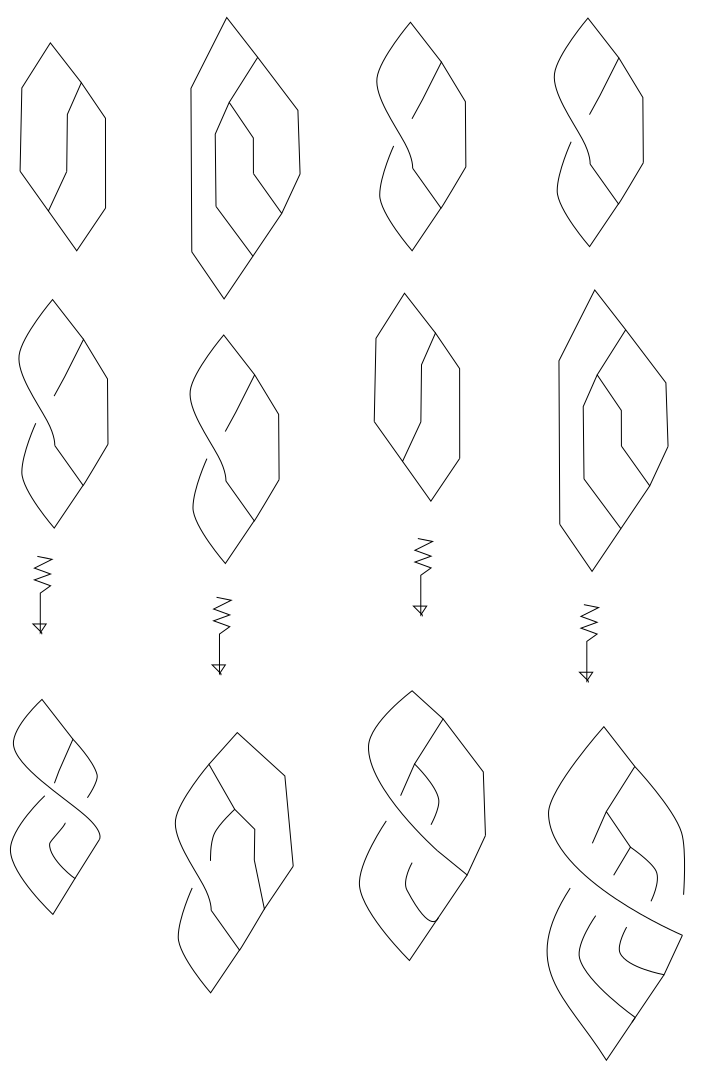}
\label{fig7}
}
\end{figure}
\end{example}

%
%
%
\begin{lemma}[]
Let $g \in BV$ be an element with the tree pair representation 
$(\mathrm{G}_{(+)},\sigma, \mathrm{G}_{(-)}).$ 
Let $\ell_{u}(\mathrm{G}_{(+)})\mapsto \ell_v(\mathrm{G}_{(-)})$ 
be a branch of $g$ and let $h$ be an element of $F$ 
and $\mathcal{K} (\cdot)$ be the number of crossings of some reduced tree pair. 
Let $h' = (h)_{[v]}.$ 
Then 
$$\mathcal{K}(g) \leq \mathcal{K}(gh')\leq \min{\{ (\mathcal{N}(g)-1)\mathcal{N}(h) +\mathcal{K}(g), \mathcal{N}(h)\mathcal{K}(g)\} }.$$ 
\label{lem312}
\end{lemma}

\begin{proof}
The former equality can be obtained from the same argument in Lemma \ref{lem25} 
and the latter one is given by multiplication of tree-braid-tree diagrams.
\end{proof}

From a few of the above consequence, 
we can see that 
for elements coming from the braid groups inside $BV,$ 
they do increase the number of carets when multiplying the torsion-free elements, 
and they also increase on the number of crossings, 
hence, 
the changes in the number of crossings are in fact very hard to control, 
when concatenating elements, 
while the changes in the number of carets when adding on elements 
is partially inherited 
from those in the original Thompson's groups $F$ inside $BV.$ 
Thus, the ``counting arguments" for constructing our paths 
will still rely on the ``non-braided" part.

\subsection{Divergence property of the geodesics} 

For braided Thompson groups, 
we would like to find a similar path 
from one geodesic to another avoiding a ball centered at the identity in their Cayley graphs
with respect to the finite generating sets, 
and we first want to find an element $g_{\mid h \mid}$ 
relying on elements in $F$
which corresponds to the element $h$ in $BV$ 
such that there is a path connecting the two elements and avoids the ball 
having radius which is the same as the word length of $h.$
\begin{proposition}
There exist constants $\delta, D > 0$ and a positive integer $Q$ 
such that the following holds. Let $g \in BV$ be an element with $\mathcal{N}(g) \geq 4 $. 
Then there is a path represented by $\omega$ 
of length at most $D | g|^4$ in the Cayley graph $\Gamma = Cay(\mathcal{G}, X)$ 
which avoids a $\delta | g | $-neighbourhood of the identity 
and which has initial vertex $g$ and terminal vertex $g\omega.$
\label{prop412}
\end{proposition}
\begin{proof}
Let the estimate be as in Lemma \ref{lem2}. 
Let $h \in BV$ whose number of leaves is at least four 
and we want to first find a subpath $\omega_1\omega_2$ 
which makes $h_2 = h\omega_1\omega_2$ supported on some interval $ [0, a_n)$.
Next, we concatenate $h_2$ by another subpath $\omega_3$ to obtain $h_3 = h_2\omega_3$ 
with the condition that 
$h_3$ commutes with $\omega_4$.
This ensures $\mathcal{N}(h_3\omega_4) \geq \mathcal{N}(h_3)$.
We finally obtain $g_{| h|} = h_4 (h_3)^{-1}$ by adding another subpath $\omega_5 = h_3^{-1}$ 
which by definition can be reduced to be $\omega_4$ 
and represents an element in $F$ inside $BV$.

Let $h$ be represented as a unique tree-braid-tree diagram 
with the notation $(\mathrm{T}_{(+)}, b, \mathrm{T}_{(-)})$.
By definition, 
$h$ will either have $\bar{\ell}_{(0)}(h) = 1$ or $\bar{\ell}_{(0)}(h) > 1$. 
When $\bar{\ell}_{(0)}(h) = 1$, 
we add a subpath $\omega_1 = x_0^2x_1^{-1}x_0^{-1}$ depending on the branches in $h$,
so that $l_{0}(h\omega_1) > 1,$
otherwise, we keep the original element $h.$

Then we add the second subpath $$\omega_2 = x_0^{-M\mathcal{N}(h_1)} x_{1}x_0^{M\mathcal{N}(h_1)} $$ 
 to $h$ or $h\omega_1$ which is a path in $F$ linear to the length of $h$,
 where $M \geq 10\frac{C_4}{C_1}$. 

The estimate of the length of $\omega_3$ is the key part in case of the braided Thompson group.
The subpath corresponding to $\omega_3$ is taken to be the shortest word representing an element 
with the following tree pair representation $(\mathrm{T}_{(-)}, b', \mathrm{T}_{(+)})$ 
such that $b'$ is taking the strand $b(0)$ to $0,$ 
i.e. $b'(b(0)) = id,$ $b'$ is not necessarily the inverse of $b.$ 
If $h \in F$, then the braid part of $\omega_3$ just represents $id.$

The path will satisfy the following properties:

\begin{enumerate}
\item $\rVert \omega_3 \lVert \leq C_3\mathcal{N}(h_1)(1+\mathcal{K}(h_1))$ 
which follows from Lemma \ref{lem2}.
       \item  $  (M-1)\mathcal{N}(h)
       \leq (M-1) \mathcal{N}(h_1\omega_2) 
       \leq \mathcal{N}(h \omega_1\omega_2\omega_3).$ 
       Moreover, 
      $  (M-1)\mathcal{K}(h)
                 \leq (M-1) \mathcal{K}(h_1\omega_2)
               \leq  \mathcal{K}(h \omega_1\omega_2\omega_3).$
              Both sequences of inequalities can be deduced from visualising the products 
              from graphical interpretation of the group elements.  
              For the increase in the number of crossings, 
              every time we concatenate a prefix in $\omega_2$ which is a word purely from generators in $F$ 
              from the left, the resulted tree-braid-tree diagram becomes larger extending below to the right side.
              Then both the number of carets an the number of crossings increase linearly 
              which ensures that ``Garside element" 
              (elements with far larger number of crossings than the number of carets \cite{MR2514382}) will not appear.
    \item $\ell_0(h_1)$ or $0^{\ell_0(h_1)} \mapsto 0^{\ell_0(h_1)}$ is a branch of $h\omega_1\omega_2\omega_3,$ 
    which follows again from the products of tree-braid-tree diagrams.
\end{enumerate}

%
%
%
%
For $h_1$ supported on $[0, \ell_{0}(h_1) ))$, 
we add on the fourth subpath $\omega_4 = x_0^{Q| h |} x_{1}^{-1} x_0^{- Q| h | +1}$
where $Q \geq 10\frac{M}{C_{1}^2}$ and then obtain $h_4 = h_3\omega_4$. 
The construction by the tree-braid-tree diagrams could be interpreted as follows: 
Noting that $h_3$ and $\omega_4$ commute on $[0, \ell_{0}(h_1) )$, 
since $h_3$ is supported on $[0, \ell_{0}(h_1) ))$ 
while $\omega_4$ from the definition fixes the braids in the same interval. 
The difference in $BV$ is that the number of crossings keep adding up linearly along the path.
Finally, we add on the last subpath $\omega_5 = h_3^{-1}$.
Since $h_4$ and $\omega_3$ commute, 
$g_{| h | } = h_4\omega_5 = h_3\omega_5h_3^{-1} = \omega_4$.

We could take $\delta = 8M\mathcal{N}(h)+3Q $ 
by purely considering the changes in the number of carets, 
so we have $ \delta | g_{| h | } | \leq | g_{| h | } \omega' |.$ 
However, 
the upper bound is given by the product of $  \mathcal{N}(g_{| h |})$ and $\mathcal{K}(g_{| h |}) $. 
Hence, the path from $h$ to $| g_{| h | } | $ has a linear lower bound 
and at most polynomial for the upper bound.
\end{proof}
\begin{lemma}
For a prefix $\omega'$ of $\omega$ representing an element in $BV,$ $|  h \omega' | \geq \delta | h |.$ 
\end{lemma}
\begin{proof}
In the case of $BV$, 
we consider both the changes in the number of carets and the number of crossings. 
When adding on the first subpath $\omega_1$, 
when $\bar{\ell}_0(h) \neq 1$, $\mathcal{N}(h) < \mathcal{N}(h\omega_1) $ 
and number of crossings $\mathcal{K}(h\omega_1) \leq \mathcal{K}(h)$
by Lemma \ref{lem39}, otherwise, we add an empty word and the length is unchanged. 
For $\omega_2,$ Lemma \ref{lem312} ensures the increase in the number of carets 
and hence the number of word length.
The word $\omega_3$ represents a subpath with the same number of carets as $h$ in the $\textbf{pcq}$ form. 
Besides we choose $\omega_3$ to be the shortest word with a non trivial braids 
with word length linear to the ones of $h.$ 
We have $ \mathcal{N}(h_1\omega_2) \leq  \mathcal{N}(h_1\omega_2\omega_3) = \mathcal{N}(h_3)  .$
Since the element represents $\omega_{3}$ having the branch $\ell_{(\sigma_{\omega_3}(0))} \mapsto \ell_{0},$
$h_3$ is supported on some $ [0, a_{|h|}).$ 
For the $\omega_4$, the argument is the same as the ones for the case when add $\omega_3.$ 
As for subpath $\omega_5$, 
it is just the inverse of path $h\omega_1\omega_2\omega_3.$
\end{proof}
\begin{theorem}
There exists constant $\delta > 0$ 
and a polynomial function $p(\cdot)$ 
such that the following holds. Let $g_1,g_2 \in \mathcal{G} $ be two elements with $\mathcal{N}(g) \geq 4.$ 
Then there is a path of length at most $p(|g_1| + |g_2|)$ in the Cayley graph $\Gamma = Cay(\mathcal{G}, X)$ 
which avoids a $\delta \min{\{| g_1| , | g_2| \} }$-neighbourhood of the identity 
and which has initial vertex $g_1$ and terminal vertex $g_2.$
\end{theorem}
\begin{proof} 
For element with small carets number the linear divergence of the geodesics 
can be deduced directly from the computation. 
Again by taking two paths we described in Proposition \ref{prop412} 
and connect them as in the Brown-Thompson group case, 
no reductions will occur.
\label{thm415}
\end{proof}
Taking the above estimate, 
we could only say that the group has linear functions as lower bounds for the divergence functions for $BV.$ 
In fact, a stronger result has been proved in \cite{Kodama:2020to} 
indicating that the divergence function of $BV$ is also bounded above by linear functions.
\begin{theorem}[Kodama]
$BV$ has linear divergence function.
\end{theorem}
Now we consider the group $BF < BV,$ which is a finitely generated group 
with finite generating set 
$$\Sigma_{BF} = \{ x_0, x_1, \alpha_{12}, \alpha_{13}, \alpha_{23}, \alpha_{24}, \beta_{12}, \beta_{13},\beta_{23},\beta_{24} \}$$ 
where $\alpha_{i,j}$s and $\beta_{i,j}$s are in fact deduced from 
a sequence of infinite generating set of the pure braid relations in the braid groups \cite{MR2384840}.
These two sequences of infinite generating sets can be written in the infinite generating sets of $BV$, 
$\tau_{i}$s and $\sigma_i$s, by the following \cite{MR3781416},
$$\alpha_{i,j} = \sigma_i\sigma_{i+1}\cdots \sigma_{j-2}\sigma_{j-1}^2 \sigma_{j-2}^{-1} \cdots \sigma_{i}^{-1}$$
$$\beta_{i,j} = \sigma_i\sigma_{i+1}\cdots \sigma_{j-2}\tau_{j-1}^2\sigma_{j-2}^{-1}\cdots \sigma_{i}^{-1}.$$
We take this finite generating set as the standard generating set when we consider the word length of $BF.$
%
%
%
%
%
%
%
%
%
\begin{lemma}
Let $g \in BF$ be an element 
with the tree-braid tree diagram pairs $(\mathrm{T}_{n(+)},\sigma, \mathrm{T}_{n(-)})$ 
and let $u\mapsto v$ be a branch of $g,$ 
$h$ be an element of $F < BF.$ 
Let $h' = (h)_{[v]}. $ 
Then $$\mathcal{N}(gh') = \mathcal{N}(g) + \mathcal{N}(h) - 1 .$$
\label{lem417}
\end{lemma}
\begin{proof}
The proof is similar to the one for Lemma \ref{lem27}.
\end{proof}
\begin{proposition}
There exist constants $\delta, D > 0$ 
such that the following holds. 
Let $g \in BF$ be an element with $\mathcal{N}(g) \geq 4 $. 
Then there is a path which can be interpreted by a word $\omega$ 
of length at most $D | g|$ in the Cayley graph 
$\Gamma = Cay(\mathcal{G}, X)$ which avoids a 
$\delta | g| $-neighbourhood of the identity 
and which has initial vertex $g$ and terminal vertex $\omega$.
\label{prop419}
\end{proposition} 
\begin{proof}
Let $h \in BF$ be an element 
whose reduced tree-braid-tree diagram pair $(\mathrm{H}_{+}, \sigma_h, \mathrm{H}_{-})$ 
has at least three carets in each tree. 
Again when $\bar{\ell}_{(0)}(h) = 1$, 
we add an extra path $\omega_1 = x_0^2x_1^{-1}x_0^{-1}$ and 
then add path 
$\omega_2 = x_0^{-M\mathcal{N}(h_1)} x_{1}x_0^{M\mathcal{N}(h_1)}$ 
followed by the shortest word 
$\omega_3$ which has the reduced tree-braid-tree diagram pair 
$(\mathrm{H}_{-}, id, \mathrm{H}_{+}).$ 
This works because $BF$ is torsion free, our paths only rely on the changes in the number of carets.
Next we take 
$\omega_4 = x_0^{Q| h |} x_{1}^{-1} x_0^{- Q| h | +1},$ 
where $Q \geq 10\frac{M}{C_{1}^2}.$ 
Since the path we are adding on represents an element in $F,$ 
the number of the crossings stays the same, thus as long as we take $Q$ large enough, 
the there will not be the influence of Garside elements \cite{MR2514382} on the distance. 
This can also be seen from the Lemma \ref{lem417}.
The estimate of the length of the path only relies on the generators $x_0$ and $x_1$ 
and hence on the number of carets.
\end{proof}
\begin{theorem}
$BF$ has linear divergence.
\label{thm420}
\end{theorem}
\begin{proof}
Provided the constant $Q$ taken is large enough, 
the concatenation of the two subpaths between the two words representing two elements in $BV$ 
that reduces the number of crossings are not going to affect the length of the linear path.
\end{proof}
\begin{remark}
According to \cite{MR2384840, BRIN_2006}, 
$\widehat{BF}$ and  $\widehat{BV}$ are also finitely generated, 
where $\widehat{BV}$ is generated by standard generating set in $F$ and the $\sigma$ generating set, 
$\widehat{BF}$ only contains the standard generating set of $F$ 
and the $\alpha$-generators induced from pure braid relation. 
In the case of $\widehat{BV}$, the estimate can be made similarly by taking five subpaths 
and concatenating them. 
The middle subpath corresponding to $\omega_3$ taken, however, 
should be modified since $\tau$-generators are not in the generating set of $\widehat{BV}.$
In $\widehat{BF},$ the word length can be estimated similarly as in the case of $BF,$ 
the path only consists of generators from the subgroup isomorphic to $F$ 
and hence the geodesics all have linear divergence.
\end{remark}
\begin{corollary}
$\widehat{BF}$ and $\widehat{BV}$ have linear divergence.
\end{corollary}
\begin{proof}
The arguments here are the analogues of the proof for $BF$ and $BV$ 
by focusing on the ``braided" part of the tree-braid-tree diagrams pair.
\end{proof}
\begin{remark}
In the previous sections, 
we considered the divergence property of the geodesics in the Cayley graph of braided Thompson groups.
We could also consider more generally 
the divergence property of the subgroups 
such as the the normal subgroup of the braided Thompson groups 
for, on one hand, 
many of the normal subgroups of the braided Thompson group are being investigated in \cite{MR3781416}, 
and they are still mysterious; 
On the other hand, 
the connection between the distortion of the finitely generated normal subgroups 
of a finitely generated group 
and the divergence of both groups has been indicated in \cite{MR3474592}, 
so we can think of the analogue in braided Thompson groups. 

In Thompson's group $F,$ 
the proper quotient subgroups of $F$ are abelian 
which somehow coincides with the linear divergence property of $F,$
whereas, 
in the case of the braided Thompson group,
the surjective map from $BF$ to $F$ provide an exact sequence, 
where the kernel corresponding to the pure braid part of the groups 
provide natural subgroups of $BF.$
However, the kernel is not finitely generated 
and there are also normal subgroups containing the commutator subgroups $[BF, BF]$ 
of which the presentation may be extremely complicated \cite{MR3781416}. 

Here divergence property gives some indication 
on the subgroup distortion of the normal subgroups inside the $BF$ and $BV.$
\end{remark}
\begin{remark}
At the point of finishing this work, 
I was told that Kodama also has a result on the divergence function of the group $BV.$
\end{remark}
\section*{Acknowledgments}
I am very grateful for the help and support from my supervisor Takuya Sakasai during 
while I am preparing for this work and during this hard time 
and I would also like to thank Tomohiro Fukaya and his student Yuya Kodama for helpful discussions.
Last but not the least, 
I would also like to express my gratitude to Sadayoshi Kojima for helpful discussion and suggestions
and for anonymous reviewers for carefully going through the manuscript 
and pointing out mistakes and inappropriate descriptions.  

\bibliographystyle{plain}

\section*{Conflict of interests}
The authors declare that they have no conflict of interest.

\section*{Data Avaliability statement}
Data sharing not applicable to this article as no datasets were generated or analysed during the current study.
\bibliography{Div20}{}


\end{document}